\documentclass[11pt,reqno,hidelinks]{amsart}
\usepackage{amsthm, amsmath,amsfonts,amssymb,euscript,hyperref,graphics,color}
\usepackage{graphicx}
\usepackage[square, comma, sort&compress, numbers]{natbib}
\usepackage{subfigure}
\usepackage{comment}
\usepackage{import}
\usepackage{tikz}
\usepackage{latexsym}

\usepackage{ulem}
\usepackage{todonotes}

\makeatother
\newtheorem{theorem}{Theorem}[section]
\newtheorem{lemma}[theorem]{Lemma}

\newtheorem{proposition}[theorem]{Proposition}

\theoremstyle{definition}

\theoremstyle{remark}
\newtheorem{remark}[theorem]{Remark}

\newcommand{\nnb}{\nonumber}

\newcommand \bel {\be\label}
\newcommand \del \partial
\newcommand \be {\begin{equation}}
\newcommand \ee {\end{equation}}
\newcommand \bes {\begin{equation*}}
\newcommand \ees {\end{equation*}}
\numberwithin{equation}{section}

\def\ub{\underline{u}}
\def\Lb{\underline{L}}
\def\Cb{\underline{C}}
\def\Eb{\underline{E}}
\def\Fb{\underline{F}}

\def\Lambdab{\underline{\Lambda}}
\def\Xb{\underline{X}}

\def\yb{\underline{y}}

\def\fb{\underline{f}}

\def\TLb{\tilde{\underline{L}}}
\def\TL{\tilde{L}}

\def\p{\partial}
\newcommand{\di}{\mathrm{d}} 
\newcommand{\D}{\mathcal{D}}
\newcommand{\T}{\mathcal{T}}

\newcommand \bei  {\begin{itemize}}
\newcommand \eei {\end{itemize}}
\allowdisplaybreaks[4]
\begin{document}
\title[relativistic string with large data]{A globally smooth solution to the relativistic string equation}

\vspace{10mm}
\author[J. Wang]{Jinhua Wang}
\address{School of Mathematical Sciences, Xiamen University, Xiamen 361005, China.}\email{ wangjinhua@xmu.edu.cn}

\author[C. Wei]{Changhua Wei}
\address{Corresponding author: Department of Mathematics, Zhejiang Sci-Tech University, Hangzhou, 310018, China.}\email{changhuawei1986@gmail.com}

\date{}

\begin{abstract}
We prove the global existence of smooth solution to the relativistic string equation in a class of data that is not small. Our solution admits the feature that the right-travelling wave can be large and the left-travelling wave is sufficiently small, and vice versa. In particular, the large-size solution exists in the whole space, instead of a null strip arising from the short pulse data. This generalizes the result of Liuli-Yang-Yu (Adv. Math. 2018) to the quasilinear setting with non-small data. In addition, in our companion paper, we are able to show the global solution here can also be seen as the non-small perturbations of the plane wave solutions.
\end{abstract}

\maketitle


\section{Introduction}\label{sec:1}
The relativistic string, which is used to study the hadrons and mechanism of their interactions, provides a clear picture of the quark confinement. It is also closely related to the dual-resonance in hadron physics, nonlinear Born-Infeld models and cosmic strings in cosmology. One may refer to the book of Barbashov, Nesterenko and Dumbrajs \cite{BND} for a more detailed introduction to the relativistic string theory.  Due to its importance in physics, the study of the relativistic string has been paid much attention and fruitful results are obtained. Let's introduce the graphic description of the relativistic string sweeping in the Minkowski spacetime $\mathbb R^{1+2}$.

Let $(t,x,y)$ be coordinates in an $1+2$-dimensional Minkowski spacetime and a graph takes the form of
\bes
y=\phi(t,x).
\ees
We call it a relativistic string or time-like extremal surface, if $\phi$ is a critical point of the area function
\be\label{I1}
I=\iint_{\mathbb R^{1+1}}\sqrt{1+\phi^2_x-\phi^2_t}dtdx,
\ee
where $\phi_x = \p_x \phi$, $\phi_t = \p_t \phi$.
The corresponding Euler-Lagrange equation reads
\be\label{I2}
\left(\frac{\phi_t}{\sqrt{1+\phi_x^2-\phi_t^2}}\right)_t-\left(\frac{\phi_x}{\sqrt{1+\phi_x^2-\phi_t^2}}\right)_x=0.
\ee
On the other hand, \eqref{I2} is a quasilinear hyperbolic system which is linearly degenerate in the sense of P. D. Lax \cite{Lax}. Namely, if we define
\bes
u :=\phi_t,\quad v:=\phi_x,
\ees
then the relativistic string \eqref{I2} is equivalent to the following quasilinear hyperbolic system
\be\label{I3}
\begin{cases}
u_t-\frac{2uv}{1+v^2}u_x+\frac{u^2-1}{1+v^2}v_x=0,\\
v_t-u_x=0,
\end{cases}
\ee
which has the eigenvalues
\be\label{I4}
\lambda_{\pm}(u,v)=\frac{-uv\pm\sqrt{1+v^2-u^2}}{1+v^2},
\ee
that are linearly degenerate.
Besides, the relativistic string also enjoys other interesting properties, such as strict hyperbolicity, boundedness of the characteristic propagation speeds, richness, etc., see \cite{Kong-Sun-Zhou} and references therein for more details.

\subsection{Previous results}
From the perspective of hyperbolic system \eqref{I3}, there were quite a lot of results for the relativistic string, including both global existence and blow-up mechanics, based on the method of characteristics. Kong and Tsuji \cite{Kong-Tsuji} found a sufficient and a necessary conditions for the global existence of classical ($C^2$) solutions to a class of $2\times 2$ hyperbolic systems with linearly degenerate characteristics. The sufficient and necessary conditions in \cite{Kong-Tsuji} are in general not the same one, however, they consist with each other in the case of relativistic string. Therefore, based on \cite{Kong-Tsuji}, Kong, Zhang and Zhou \cite{Kong-Zhang, Kong-Zhang-Zhou} addressed a sufficient and necessary condition for the global existence of a closed relativistic string evolving in Minkowski space $\mathbb R^{1+n}$ and numerically showed the appearance of topological singularities if this condition was not satisfied. Furthermore, Nguyen and Tian \cite{Nguyen-Tian} showed that the timelike maximal cylinders in $\mathbb R^{1+2}$ always develop singularities in finite time. Related to \cite{Kong-Tsuji}, Kong, Wei and Zhang \cite{Kong-Wei-Zhang, Kong-Wei} showed a blow-up phenomenon for the one dimensional isentropic Chaplygin gas. In fact, a ``cusp'' type singularity was exhibited in \cite{Kong-Wei-Zhang}, for which the density of the Chaplygin gas itself tends to infinity. The relation between the relativistic string and Chaplygin gas was stated by Bordeman and Hoppe \cite{Hoppe}. These results verified a famous conjecture proposed by A. Majda \cite{Majda}, which says ``The Cauchy problem of the quasilniear hyperbolic system with totally linearly degenerate characteristics admits a uniquely global solution, unless the solution itself blows up in finite time''.

Alternatively, the relativistic string equation represented by \eqref{I2} is an $1+1$-dimensional quasilinear wave equation satisfying the double null condition. 
From this point of view, it was initiated with a small-data global existence given by Lindblad \cite{Lindblad}, and was later generalized in Wong \cite{Wong}. Recently, Luli, Yang and Yu \cite{Luli-Yang-Yu} introduced a new type of weighted energy estimate to prove the small-data global existence of a semilinear wave equation with null condition in the Minkowski spacetime $\mathbb R^{1+1}$, where the decay mechanism comes from the spatial separation of two families of wave packet (the right-travelling and left-travelling ones) after a sufficiently long time. It consequently provides a new way to see the one-dimensional hyperbolic system, and was extended by Zha \cite{Zha,Zha1} to quasilinear wave equations and one-dimensional quasilinear hyperbolic system.  Later, Abbrescia and Wong \cite{Wong1} employed a geometric approach to consider a class of $1+1$-dimensional variational quasilinear wave equations (including the relativistic string equation) satisfying the null condition of Klainerman, and got a globally $C^2$ solution with data that are not necessarily small. Specifically, Abbrescia and Wong \cite{Wong1} achieved the globally nonlinear stability of a class of simple wave solutions under small perturbations. We remark that different from \cite{Luli-Yang-Yu}, Abbrescia and Wong \cite{Wong1} introduced a dynamical double null coordinate system, with which they turned the quasilinear equations into a semilinear system, and solved the perturbed system directly. Eventually, his argument was closed by showing that the dynamical coordinates are $C^1$ regular in the entire spacetime. These ideas inspired by works on formation of shock in fluid somehow happen to coincide with the ones in characteristic method \cite{Kong-Tsuji}.

For relativistic string on non-flat spacetimes, there are, for instance, the global result on the Schwarzschild background \cite{He-Huang-Kong}, and the global stability of travelling wave solutions in the de sitter spacetime proved by He, Huang and Wei \cite{He-Huang-Wei}. Besides, Liu and Zhou \cite{Liu-Zhou} concerned the initial-boundary value problems of relativistic string.

In this paper, we take the method of weighted energy estimates inspired by \cite{Luli-Yang-Yu} and \cite{Wang-Wei} to study the globally smooth solutions to the relativistic string equation \eqref{I2} with non-small data. We emphasis that our method is robust. In our companion paper \cite{Wang-Wei-string-2}, we show the global stability of the plane wave solutions to \eqref{I2} within a class of non-small perturbations, for which the argument presented here plays a crucial role. Besides, we reformulate the relativistic string equation in \cite{Wang-Wei-string-2} based on another interesting work of Abbrescia and Wong \cite{Abbrescia-Wong}.

\subsection{Main theorem}
We consider the Cauchy problem of the relativistic string equation \eqref{I2} with initial data given by
\be\label{I7}
(\phi, \, \phi_t)|_{t=0}=(F(x), \, G(x)),
\ee
where $F(x)$ and $G(x)$ are smoothly real-valued functions.

\begin{theorem}\label{main-theorem}
Fix $0<\gamma<1$ and a constant $I$ that can be arbitrarily large. Suppose there exist smooth functions $f(x), \, \fb(x) \in C^\infty(\mathbb R)$ satisfying
\bes
\int_{\mathbb R}(1+|x|)^{2+2\gamma} \left( |f^{(k)}(x)|^2 + |\fb^{(k)}(x)|^2 \right) dx \leq I, \quad \text{for all} \,\, k \in \mathbb{Z}_{\geq 0},
\ees
such that our data $(F(x), G(x))$ obey
\bes
G(x)+F^\prime(x)   = \delta f(x), \quad  G(x) - F^\prime(x) = \fb(x),
\ees
where $\delta$ is a constant. Then if $\delta$ is small enough (depending on the data $I$), the Cauchy problem \eqref{I2}, \eqref{I7} admits a unique and globally smooth solution in $C^{\infty}(\mathbb R^{+}\times \mathbb R)$. Here we denote $f^{(k)} (x) := \frac{\di^{k}}{\di x^{k}} f(x)$, and similar for $\fb^{(k)}(x)$.
\end{theorem}
\begin{remark}
The data assumption in Theorem \ref{main-theorem} will ensure that during the propagations, the energy of the left-travelling wave $(\p_t + \p_x) \phi$ is small while the energy of the right-travelling wave $(\p_t - \p_x) \phi$ can be as large as $I$. Alternatively, if initially the right travelling wave is small and the right-travelling wave is large,  a similarly global result holds true as well.
\end{remark}

\begin{remark}
Associated to the data $(F(x), \, G(x))$, we define
\be\label{I8}
\Lambda_{\pm}(x)=\frac{-F^{\prime}(x)G(x)\pm\sqrt{1+(F^{\prime}(x))^2-G^2(x)}}{1+(F^{\prime}(x))^2}.
\ee
That is, $\Lambda_{\pm}(x)$ are the eigenvalues \eqref{I4} restricted on the initial surface.
The main result of \cite[Theorem 2.2]{Kong-Tsuji} indicated that the relativistic string equation \eqref{I2} admit a unique and globally $C^2$ solution, provided that the initial data satisfy the following assumptions: There exist two constants $\Lambda_{*}$ and $\Lambda^{*}$ such that for all $x\in\mathbb R$
\be\label{I9}
\Lambda_{*}\leq \Lambda_{\pm}(x)\leq \Lambda^{*}\quad\text{and}\quad \Lambda_{-}(x)<\Lambda_{+}(x)
\ee
and for every fixed $x_2\in\mathbb R$ and $x_1\in(-\infty,x_2)$, it holds that
\be\label{I10}
\Lambda_{-}(x_1)<\Lambda_{+}(x_2).
\ee

Now, under our data assumptions, 
\bes
\Lambda_{\pm}(x)
=\frac{\fb^2(x)-\delta^2f^2(x)\pm4\sqrt{1-\delta \fb(x)f(x)}}{4+\fb^2(x)+\delta^2f^2(x)-2\delta \fb(x)f(x)}.
\ees
It is easy to verify that \eqref{I9} and \eqref{I10} hold true for our data if $\delta$ is sufficiently small, which means that our data satisfy the sufficient condition of Kong and Tsuji \cite{Kong-Tsuji}.

For the relativistic string equation, \eqref{I9}-\eqref{I10} is the sufficient and necessary condition for the global existence \cite{Kong-Zhang-Zhou}.
We remark that, in the case of small data, it necessarily obeys \eqref{I9} and \eqref{I10}.


Different from \cite{Kong-Tsuji} and \cite{Wong1}, we obtain the $C^{\infty}$ solution instead of the $C^2$ one.
\end{remark}

Essentially, we extend the result of \cite{Luli-Yang-Yu} to a quasilinear wave equation in a non-small setting. Our proof takes advantage of the weighted energy estimates \cite{Luli-Yang-Yu}, geometrically adapted multipliers \cite{Wang-Wei} and hierarchy of energy estimates \cite{Christodoulou, Wang-Wei, Wang-Yu, Wang-Yu1}. The multiplier vector fields $\underline{\Lambda}(\underline{u})L$ and $\Lambda(u)\underline{L}$ used in \cite{Luli-Yang-Yu} are replaced by the following dynamical ones
\be\label{multiplier-intro}
\Lambdab(u) (L + |L\phi|^2 \Lb) \quad \text{and} \quad \Lambda(u) (\Lb + |\Lb \phi|^2 L),
\ee
where $L=\p_t + \p_x$ and $\Lb=\p_t - \p_x$. As what is illustrated in \cite{Wang-Wei}, the two vector fields in \eqref{multiplier-intro} are non-spacelike with respect to the geometry of string. In contrast to the small data setting in \cite{Luli-Yang-Yu}, a hierarchy of energy estimate is needed to close the argument, because in our energy strategy, the left-travelling wave $L\phi$ is small, while the right-travelling waves $\Lb \phi$ is non-small.

Comparing with \cite{Wang-Wei} where the globally large solution of the $1+n$-dimensional ($n=2, 3$) relativistic membrane equation with short pulse data was obtained, here our non-small data are defined on the whole space instead of the short pulse one confined in a thin domain. As a result, the large-size solution or the right-travelling wave distributes in the whole $\mathbb{R}^{1+1}$, rather than a thin null strip as in \cite{Wang-Wei}. Of course,  it holds that in \cite{Wang-Wei}, the higher order derivatives it takes, the larger energy norm it becomes. But in our present situation, each order energy (of $\delta$ size or non-small size) is consistent.

We arrange our paper as follows: In Section \ref{sec:2}, we give some preparations for the proof of the main theorem including the notations, energy scheme and the vector fields. Section \ref{sec:3} is the main body of the paper, in which we prove the main Theorem \ref{main-theorem} by energy estimates. In the last section, we describe the initial data that satisfy the smallness requirement in the energy estimates.

\section{Preliminaries}\label{sec:2}
In this section, we give some preliminaries for proving the main theorem, including the geometry of the relativistic string, energy scheme and vector fields.
\subsection{Reformulation of the relativistic string equation geometrically}
Let $\eta_{\mu \nu}$ be the standard Minkowski metric. We recall the null foliations in the Minkowski spacetime. The null coordinates are
\begin{equation*}
\ub=\frac{t+x}{2},\quad u =\frac{t-x}{2},
\end{equation*}
Their corresponding null vectors are defined by
\begin{equation*}
L=\partial_{\underline{u}}=\partial_{t}+\partial_{x},\quad \underline{L}=\partial_{u}=\partial_{t}-\partial_{x},
\end{equation*}
which obey
\bes
\eta(L,L)=\eta(\underline{L},\underline{L})=0,\quad\eta(L,\underline{L})=-2.
\ees
It is straightforward to verify
\bes
Lu=\underline{L}\underline{u}=0,\quad L\underline{u}=\underline{L}u=1.
\ees

Based on the geometry of the relativistic string, similar to the relativistic membrane in \cite{Wang-Wei}, we can rewrite the relativistic string equation as
\begin{equation}\label{2.7}
\Box_g \phi = \frac{1}{\sqrt{g}}\partial_{\alpha}\left(\sqrt{g}g^{\alpha\beta}\partial_{\beta}\phi\right)=0.
\end{equation}
Combining with the following gauge condition
\begin{equation}\label{wave-coordinate}
\partial_{\alpha}\left(\sqrt{g}g^{\alpha\beta}\right)=0,\quad \alpha,\,\beta=0,1,
\end{equation}
we obtain
\begin{equation}\label{2.8}
g^{\alpha\beta}\partial_{\alpha}\partial_{\beta}\phi=0,
\end{equation}
where the associated metric takes the following form
\begin{equation}\label{2.9}
g_{\alpha\beta}=\eta_{\alpha\beta}+\partial_{\alpha}\phi\partial_{\beta}\phi,\quad g^{\alpha\beta}=\eta^{\alpha\beta}-\frac{\partial^{\alpha}\phi\partial^{\beta}\phi}{g}.
\end{equation}
Note that, $g^{\alpha\beta}$ denotes the inverse of $g_{\alpha\beta}$
and the determinant of $g_{\alpha \beta}$ is given by
\begin{equation}\label{2.10}
g \doteq1+Q(\phi,\phi)=1-L\phi\Lb\phi.
\end{equation}
\begin{remark}
The relativistic string is also called a time-like extremal surface if $g>0$.
The determinant $g=1-\phi^2_t+\phi^2_x >0$ is nothing but the time-like condition.
\end{remark}

\subsection{Energy scheme}
We can define the energy momentum tensor corresponding to \eqref{2.7}
\begin{equation}\label{2.13}
\T^{\alpha}_{\beta} [\varphi] \doteq g^{\alpha \mu}\partial_{\mu}\varphi\partial_{\beta}\varphi-\frac{1}{2} g^{\mu\nu}\partial_{\mu}\varphi\partial_{\nu}\varphi \delta^{\alpha}_{\beta},
\end{equation}
where $\delta^{\alpha}_{\beta}$ denotes the Dirac function.

For any vector field $\xi$, the associated current $P^{\alpha}$ is defined as
\begin{equation}\label{2.14}
P^{\alpha}= P^{\alpha}[\varphi,\xi] \doteq \T^{\alpha}_{\beta} [\varphi] \cdot \xi^{\beta}.
\end{equation}
The divergence of $P^{\alpha}$ could be calculated as follow, see \cite{Wang-Wei} for the detailed proof.
\begin{lemma}\label{lem:2.4}
The energy current $P^{\alpha}$ satisfies
\begin{equation}\label{eq-div-P}
\frac{1}{\sqrt{g}}\partial_{\alpha}(\sqrt{g}P^{\alpha})= \Box_{g(\p\varphi)}\varphi \cdot \xi \varphi + \T^{\alpha}_{\beta} [\varphi] \cdot \partial_{\alpha}\xi^{\beta}
-\frac{1}{2\sqrt{g}}\xi(\sqrt{g}g^{\gamma\rho})
\partial_{\gamma}\varphi\partial_{\rho}\varphi.
\end{equation}
\end{lemma}

We also introduce the $(0,2)$-energy momentum tensor $\T_{\alpha \beta} \doteq g_{\alpha \sigma} \T^{\sigma}_{\beta}$, where $\T^{\sigma}_{\beta}$ is the $(1,1)$-energy momentum tensor defined in \eqref{2.13}, so that we can reformulate the current as
\begin{equation}\label{def-energy-xi-u-ub}
-P^{u}[\varphi,\xi] = \T[\varphi](-Du, \xi), \quad -P^{\ub}[\varphi,\xi] = \T[\varphi](-D \ub, \xi),
\end{equation}
and
\begin{equation}\label{def-energy-xi-t}
-P^{t}[\varphi,\xi] = \T[\varphi](-Dt, \xi).
\end{equation}
For any function $f$, we always denote $Df \doteq g^{\alpha\beta}\partial_{\alpha}f\partial_{\beta}$ the gradient with respect to $g_{\alpha \beta}$. We will omit the $[\varphi]$ in $\T[\varphi]_{\alpha \beta}, \, \T^\alpha_\beta [\varphi]$ if there is no confusion.

Following \cite{Luli-Yang-Yu}, we define the out-going null segment
$$C^\tau_{u_0} \doteq\Big \{(t, x) \big| \frac{t-x}{2} = u = u_0, \, 0 \leq t \leq \tau \Big\},$$
and the in-coming null segment
$$\Cb^\tau_{\ub_0} \doteq \Big\{(t, x) \big| \frac{t+x}{2} = \ub = \ub_0, \, 0 \leq t \leq \tau \Big\}.$$
The spacetime region on the right of $C^\tau_{u_0}$ is,
$$\D^{+}_{\tau, u_0} \doteq \Big\{(t, x) \big| 0 \leq t \leq \tau, \,   \frac{t-x}{2} = u \leq u_0 \Big\},$$ and the spacetime region on the left of of $\Cb^\tau_{\ub_0}$ is 
$$\D^{-}_{\tau, \ub_0} \doteq \Big\{(t, x) \big| 0 \leq t \leq \tau, \,   \frac{t+x}{2} = \ub \leq \ub_0\Big \}.$$
We also define
\bes
\D_{\tau} \doteq \{(t, x) \big| 0 \leq t \leq \tau \}.
\ees
Then $\D_\tau$ can be foliated by $\{C^\tau_u\}_{u \in \mathbb{R}}$ and $\{ \Cb^\tau_{\ub}\}_{\ub \in \mathbb{R}}$, and $\D^{-}_{\tau, \ub} \subset \D_\tau$, $\D^{+}_{\tau, u} \subset \D_\tau$.
There are corresponding spatial segments
\begin{align*}
\Sigma^{+}_{\tau, u_0} \doteq \Big\{(t, x) \big| t =\tau, \, \frac{t-x}{2} = u \leq u_0\Big \}, \\
\Sigma^{-}_{\tau, \ub_0} \doteq \Big\{(t, x) \big| t =\tau, \, \frac{t+x}{2} = \ub \leq \ub_0\Big \}.
\end{align*}
And we denote $$\Sigma_\tau \doteq \{(t, x) \big| t=\tau\}.$$ There is of course, $\Sigma^{+}_{\tau, u}  \subset \Sigma_\tau$, and  $\Sigma^{-}_{\tau, \ub}  \subset \Sigma_\tau$.

Applying the divergence theorem to these domains and noting \eqref{def-energy-xi-u-ub}-\eqref{def-energy-xi-t}, we obtain the following energy identities.

On $\D^{-}_{\tau, \ub},$ we have
\begin{equation}\label{energy-id-1}
\begin{split}
& \int_{\Sigma^{-}_{\tau, \ub}} \T(-Dt, \xi) \sqrt{g} \di x+ \int_{\Cb^\tau_{\ub}} \T(-D \ub, \xi) \sqrt{g} \di u \\
={}& \int_{\Sigma^{-}_{0, \ub}} \T(-Dt, \xi)\sqrt{g} \di x -\iint_{\D^{-}_{\tau, \ub}}\p_{\alpha}(\sqrt{g}P^{\alpha}) \di t\di x,
\end{split}
\end{equation}
and on $\D^{+}_{\tau, u}$, we have
\begin{equation}\label{energy-id-2}
\begin{split}
& \int_{\Sigma^{+}_{\tau, u}} \T(-Dt, \xi) \sqrt{g} \di x+ \int_{C^\tau_{u}} \T(-D u, \xi) \sqrt{g} \di u \\
={}& \int_{\Sigma^{+}_{0, u}} \T(-Dt, \xi)\sqrt{g} \di x -\iint_{\D^{+}_{\tau, u}}\p_{\alpha}(\sqrt{g}P^{\alpha}) \di t \di x.
\end{split}
\end{equation}

\subsection{Vector fields}
In application, we will use multiplier vector fields taking the following forms
\bel{def-Z-Zb}
Z = X+y \Lb, \quad \underline{Z} = \Xb + \yb L,
\ee
with $X, \, \Xb$ being the ones introduced in \cite{Luli-Yang-Yu},
\begin{align*}
X &= \Lambdab(\ub) L, \quad \Xb = \Lambda(u) \Lb,
\end{align*}
and $y$, $\yb$ are functions depending on the unknown which will be given later. If we define $$a(x) = (1+x^2)^{1+\gamma},$$ then the $\Lambdab(\ub), \, \Lambda (u)$ are $$\Lambdab(\ub) = a(\ub), \quad \Lambda (u) = a(u).$$

Given any vector field $\xi$, we define the deformation tensor with respect to the Minkowski metric $\eta$ by:
$\Pi^\xi_{\alpha \beta} \doteq \frac{1}{2} \mathcal{L}_\xi \eta_{\alpha \beta}$. Then  \cite{Luli-Yang-Yu}
\bes
\Pi^X_{\alpha \beta} = \frac{1}{2} \Lambdab^\prime (\ub) \eta_{\alpha \beta}, \quad \Pi^{\Xb}_{\alpha \beta} = \frac{1}{2} \Lambda^\prime (u) \eta_{\alpha \beta}.
\ees
By virtue of \eqref{eq-div-P}, we need to calculate
\begin{align*}
\T^{\alpha}_{\beta} [\varphi] \cdot \p_{\alpha} \underline{Z}^{\beta} &= \T^{\alpha}_{\beta} [\varphi] \cdot \left( \p_{\alpha} \Xb^{\beta} + \p_\alpha (\yb L)^\beta  \right) \\
& =  \T^{\alpha}_{\beta} [\varphi] \eta^{\beta \gamma} \Pi^{\Xb}_{\alpha \gamma} +  \T^{\alpha}_{\beta} [\varphi]  \p_\alpha (\yb L)^\beta \\
& =  \frac{1}{2} \Lambda^\prime (u) \T^{\alpha}_{\beta} [\varphi] \eta^{\beta \gamma} \eta_{\alpha \gamma} +  \T^{\alpha}_{\beta} [\varphi]  \p_\alpha (\yb L)^\beta.
\end{align*}
Since in $1+1$ dimension,
\bes
\T^{\alpha}_{\beta} [\varphi] \eta^{\beta \gamma} \eta_{\alpha \gamma}  = \T^{\alpha}_{\alpha} [\varphi] = 0,
\ees
the first term coming from the deformation tensor $\Pi^{\Xb}$ (or $\Pi^X$) vanishes, we therefore derive
\begin{align}
\T^{\alpha}_{\beta} [\varphi] \cdot \p_{\alpha} \underline{Z}^{\beta} & = \T^{\alpha}_{\beta} [\varphi]  \p_\alpha(\yb L)^\beta = L^\beta \T^{\alpha}_{\beta} [\varphi]  \p_\alpha \yb, \label{deformation-Zb}\\
\T^{\alpha}_{\beta} [\varphi] \cdot \p_{\alpha} Z^{\beta} & = \T^{\alpha}_{\beta} [\varphi]  \p_\alpha (y \Lb)^\beta = \Lb^\beta \T^{\alpha}_{\beta} [\varphi]  \p_\alpha y. \label{deformation-Z}
\end{align}

\section{The proof of Theorem \ref{main-theorem}}\label{sec:3}
In this section, we prove the main Theorem \ref{main-theorem}.
\subsection{Multipliers}
There are two natural vector fields $D u$ and $D \ub$
\begin{equation}\label{grad-u-ub-expansion-1}
\begin{split}
D u &\doteq g^{\mu\nu}\p_{\mu}u \p_{\nu}= -\frac{1}{2} L - \frac{L\phi \Lb\phi}{4g} L  - \frac{(L\phi)^2}{4g} \Lb, \\
D \ub &\doteq g^{\mu\nu}\p_{\mu}\ub \p_{\nu}=  -\frac{1}{2} \Lb - \frac{L\phi \Lb\phi}{4g} \Lb  - \frac{(\Lb\phi)^2}{4g} L,
\end{split}
\end{equation}
which are non-spacelike with respect to the geometry $g_{\alpha \beta}$ of the relativistic string, for there is
\begin{align*}
g(Du, Du) = g^{u u} & = -\frac{(L \phi)^2}{4g}, \\
g(D \ub, D \ub) = g^{\ub \ub}&  = -\frac{(\Lb \phi)^2}{4g}.
\end{align*}
Motivated by these vector fields, we introduce the following multipliers \cite{Wang-Wei} (after ignoring the lower order terms in \eqref{grad-u-ub-expansion-1})
\bel{def-TL-TLb}
\TL= \Lambdab(\ub) (L + |L\phi|^2 \Lb), \quad \TLb = \Lambda(u) (\Lb + |\Lb \phi|^2 L).
\ee
That is, in \eqref{def-Z-Zb}, we take
\bel{def-TL-TLb1}
y = \Lambdab(\ub) |L\phi|^2, \quad \yb = \Lambda(u) |\Lb \phi|^2.
\ee
We see that
\begin{align*}
g(\TL, \TL) & =  \Lambdab^2 (\ub) |L \phi|^2 (-3 + 2L \phi \Lb \phi+|L\phi|^2|\Lb \phi|^2), \\
 g(\TLb, \TLb) & =\Lambda^2 (u) |\Lb \phi|^2 (-3 + 2L \phi \Lb \phi+|L\phi|^2|\Lb \phi|^2).
\end{align*}
This implies that $\TL$ and $\TLb$ are non-spacelike, provided that $-3 + 2L \phi \Lb \phi+|L\phi|^2|\Lb \phi|^2<0$.
\begin{remark}
Due to the quasilinear feature of relativistic string equation, the multipliers used in \cite{Luli-Yang-Yu} are now no longer not time-like with respect to the geometry $g_{\alpha \beta}$. The multipliers \eqref{def-TL-TLb} we choose are inspired by our work on large data problem of the higher dimensional relativistic membranes \cite{Wang-Wei}. As in \cite{Luli-Yang-Yu}, here we need the extra weights $\Lambdab(\ub)$ and $\Lambda(u)$ to gain the ``decay'' properties of the solution.
\end{remark}


We calculate $\T^{\alpha}_{\beta} [\varphi] \cdot \partial_{\alpha}\TL^{\beta}$, and $ \T^{\alpha}_{\beta} [\varphi] \cdot \partial_{\alpha}\TLb^{\beta}$ which will be needed in the energy estimates later.

For convenience, we denote
\bel{def-tilde-Q}
\tilde Q (\varphi, \psi) = g^{\mu\nu}\partial_{\mu}\varphi\partial_{\nu}\psi.
\ee
Then
\bes
\tilde Q (\varphi, \varphi) = Q(\varphi, \varphi) - q(\varphi, \varphi), \quad \text{with} \quad q(\varphi, \varphi) \doteq \frac{1}{1+Q(\phi, \phi)} (\p^\mu \phi \p_\mu \varphi)^2.
\ees
We recall that $Q(\varphi, \psi) = - L\varphi \Lb \psi$.

\begin{lemma}\label{lem-deformation}
We have
\begin{align*}
\T^{\alpha}_{\beta} [\varphi]  \p_{\alpha} \TLb^{\beta}  &= \left( - \frac{1}{2} |L \varphi|^2 - \frac{1}{4 g} L \phi \Lb \phi |L \varphi|^2  - \frac{1}{4 g} |L \phi|^2 L \varphi \Lb \varphi  \right)  (\Lambda^\prime (u) |\Lb \phi|^2 + 2 \Lambda (u) \Lb^2 \phi \Lb \phi ) \\
& - \frac{1}{8 g} \left( |\Lb \phi|^2 |L \varphi|^2 -  |L \phi|^2 |\Lb \varphi|^2  \right) \cdot  2 \Lambda (u) L \Lb  \phi \Lb \phi,
\end{align*}
and
\begin{align*}
\T^{\alpha}_{\beta} [\varphi] \p_{\alpha} \TL^{\beta} &= \left( - \frac{1}{2} |\Lb \varphi|^2 - \frac{1}{4 g} L \phi \Lb \phi |\Lb \varphi|^2  - \frac{1}{4 g} |\Lb \phi|^2 L \varphi \Lb \varphi  \right)  ( \Lambdab^\prime (\ub) |L \phi|^2 + 2 \Lambdab (\ub) L^2 \phi L \phi ) \\
& + \frac{1}{8 g} \left( |\Lb \phi|^2 |L \varphi|^2 -  |L \phi|^2 |\Lb \varphi|^2  \right) \cdot  2 \Lambdab (\ub) L \Lb  \phi L \phi.
\end{align*}
\end{lemma}

\begin{proof}
We recall that $\TL = X + y \Lb, \, \TLb = \Xb + \yb L$, with $y, \, \yb$ being given by \eqref{def-TL-TLb1}. As a consequence of \eqref{deformation-Zb}-\eqref{deformation-Z}, there are
\begin{align*}
\T^{\alpha}_{\beta} [\varphi] \cdot \p_{\alpha} \TLb^{\beta} & = L^\beta \T^{\alpha}_{\beta} [\varphi]  \p_\alpha \yb, \\
\T^{\alpha}_{\beta} [\varphi] \cdot \p_{\alpha} \TL^{\beta} & = \Lb^\beta \T^{\alpha}_{\beta} [\varphi]  \p_\alpha y.
\end{align*}
A straightforward calculation shows
\begin{align*}
L^\beta \T^{\alpha}_{\beta} [\varphi]  \p_\alpha \yb & = \T_{\ub}^u  [\varphi] \cdot \Lb (\Lambda (u) |\Lb \phi|^2) + \T_{\ub}^{\ub}  [\varphi] \cdot L (\Lambda (u) |\Lb \phi|^2) \\
&= g^{u \alpha} \p_\alpha \varphi \p_{\ub} \varphi \Lb (\Lambda (u) |\Lb \phi|^2) + \left( g^{\ub \alpha} \p_\alpha \varphi \p_{\ub} \varphi - \frac{1}{2} \tilde Q(\varphi, \varphi) \right) L (\Lambda (u) |\Lb \phi|^2).
\end{align*}
Notice that
\begin{align*}
 g^{\ub \alpha} \p_\alpha \varphi \p_{\ub} \varphi & =  g^{\ub u} \p_u \varphi \p_{\ub} \varphi + g^{\ub \ub} (\p_{\ub} \varphi)^2, \\
 \tilde Q(\varphi, \varphi) & =  g^{u u} (\p_{u} \varphi)^2 + 2 g^{\ub u} \p_u \varphi \p_{\ub} \varphi + g^{\ub \ub} (\p_{\ub} \varphi)^2, \\
 g^{\ub \alpha} \p_\alpha \varphi \p_{\ub} \varphi - \frac{1}{2} \tilde Q(\varphi, \varphi) & = \frac{1}{2} (g^{\ub \ub} (\p_{\ub} \varphi)^2 - g^{u u} (\p_{u} \varphi)^2).
\end{align*}
It then follows that
\begin{align*}
L^\beta \T^{\alpha}_{\beta} [\varphi]  \p_\alpha \yb &= \left( - \frac{1}{2} |L \varphi|^2 - \frac{1}{4 g} L \phi \Lb \phi |L \varphi|^2  - \frac{1}{4 g} |L \phi|^2 L \varphi \Lb \varphi  \right)  (\Lambda^\prime (u) |\Lb \phi|^2 + 2 \Lambda (u) \Lb^2 \phi \Lb \phi ) \\
& - \frac{1}{8 g} \left( |\Lb \phi|^2 |L \varphi|^2 -  |L \phi|^2 |\Lb \varphi|^2  \right) \cdot  2 \Lambda (u) L \Lb  \phi \Lb \phi.
\end{align*}

The identity for $\T^{\alpha}_{\beta} [\varphi]  \p_{\alpha} \TLb^{\beta}$ can be derived in the same way.

\end{proof}

\subsection{Energy estimates} Now we are ready to do the energy estimates in this subsection.

\subsubsection{Bootstrap arguments.}
Let
\bel{def-phi-k}
\phi_k \doteq \partial^k\phi,
\ee
where we use $\partial^k\phi$ to denote derivatives of $|k|$ order: $\partial^k\phi =\sum_{|k_1|+|k_2| = |k|}\partial_t^{k_1}\partial_x^{k_2}\phi$.

Define the energies
\begin{subequations}
\begin{align}
E^2_{[k+1]}(u,t) & =\int_{\Sigma_{t,u}^{+}} \Lambdab(\ub)  |L \phi_{k}|^2 \sqrt{g} \di x, \label{def-E1} \\
\Eb^2_{[k+1]}(\ub, t) & =\int_{\Sigma_{t,\ub}^{-}} \Lambda (u) |\Lb \phi_{k}|^2  \sqrt{g} \di x, \label{def-Eb1}\\
F^2_{[k+1]}(u,t)&=\int_{C_{u}^{t}}\Lambdab(\ub)|L\phi_k|^2\sqrt{g}d\ub,\label{def-F}\\
\Fb^2_{[k+1]}(\ub,t)&=\int_{C_{\ub}^{t}}\Lambda(u)|\Lb\phi_k|^2\sqrt{g}du, \label{def-Fb}
\end{align}
\end{subequations}
and
\begin{subequations}
\begin{align}
E^2_{[k+1]}(t) &=  \sup_{u \in \mathbb{R}} E^2_{[k+1]}(u, t)  =\int_{\Sigma_{t}} \Lambdab(\ub)  |L \phi_{k}|^2 \sqrt{g} \di x, \label{def-E} \\
\Eb^2_{[k+1]}(t) &=  \sup_{\ub \in \mathbb{R}} \Eb^2_{[k+1]}(\ub, t) =\int_{\Sigma_{t}} \Lambda (u) |\Lb \phi_{k}|^2  \sqrt{g} \di x, \label{def-Eb}
\end{align}
\end{subequations}
These energies are essentially the principle part of the ones appearing on the left hand side of the energy identities \eqref{energy-id-1} with $\xi=\TLb$, \eqref{energy-id-2} with $\xi = \TL$, respectively. This will be shown in Lemma \ref{lemma-energy-formula}.
 We also drop the bracket in the subscript to denote the inhomogeneous energy
\begin{align}
E^2_{N+1}(t) & =\sum_{k=0}^{N} E^2_{[k+1]}(t) \label{def-E-IH}.
\end{align}
Analogous notations apply to $\Eb^2_{N+1}, \, F^2_{N+1}, \, \Fb^2_{N+1}, \cdots$ as well.
Given any $N \in \mathbb{N}$, $N \geq 2$, we assume that
\begin{align}
& \Eb^2_{N+1}(t )+ \sup_{\ub \in \mathbb{R}} \Fb^2_{N+1}(\ub, t) \leq M^2, \label{bt-Eb-Fb}\\
& E^2_{N+1}(t) + \sup_{u \in \mathbb{R}} F^2_{N+1}(u, t) \leq \delta^2 M^2, \label{bt-E-F}
\end{align}
for all $t \geq 0$. Here $M$ is a large constant to be determined later.
\begin{remark}
In contrast to \cite{Luli-Yang-Yu} where all the energies are bounded at the same level, we have a hierarchy between the $E, \, F$ and $\Eb, \, \Fb$ energies. Namely, in the continuity argument, the bounds for $E, \, F$ energies will be improved preferentially, and we will need this information to further improve the energy bounds of $\Eb, \, \Fb$. 
\end{remark}

As a consequence, we have the $L^\infty$ bound and the weighted Sobolev inequalities.
\begin{lemma}\label{lem-sobolev}
 With the bootstrap assumptions \eqref{bt-Eb-Fb} and \eqref{bt-E-F}, we have, when $k\leq N-1$,
\bel{Sobolev-L-infty}
|L\phi_k| \lesssim \frac{\delta M}{\Lambdab^{\frac{1}{2}}(\ub)}, \quad |\Lb \phi_k| \lesssim \frac{M}{\Lambda^{\frac{1}{2}}(u)},
\ee
and
\begin{align}
\Big\| \frac{\Lambdab^{\alpha}(\ub)}{\Lambda^{\beta} (u)} L \phi_k \Big\|_{L^\infty(\Sigma^+_{t, u_0})} \lesssim \Big\| \frac{\Lambdab^{\alpha}(\ub)}{\Lambda^{\beta} (u)} L \phi_k\Big\|_{L^2(\Sigma^+_{t, u_0})} + \Big\| \frac{\Lambdab^{\alpha}(\ub)}{\Lambda^{\beta} (u)} L \p_x \phi_k\Big\|_{L^2(\Sigma^+_{t, u_0})}, \label{Sobolev-L-infty-weight-L} \\
\Big\| \frac{\Lambda^{\alpha}(u)}{\Lambdab^{\beta} (\ub)}  \Lb \phi_k \Big\|_{L^\infty(\Sigma^-_{t, \ub_0})} \lesssim \Big\| \frac{\Lambda^{\alpha}(u)}{\Lambdab^{\beta} (\ub)}  \Lb \phi_k \Big\|_{L^2(\Sigma^-_{t, \ub_0})} + \Big\| \frac{\Lambda^{\alpha}(u)}{\Lambdab^{\beta} (\ub)}  \Lb \p_x \phi_k\Big\|_{L^2(\Sigma^-_{t, \ub_0})}, \label{Sobolev-L-infty-weight-Lb}
\end{align}
for all $u_0, \ub_0 \in \mathbb{R}$ and $\alpha, \, \beta \in \mathbb{R^+}$.
\end{lemma}
\begin{proof}
As shown in \cite{Luli-Yang-Yu}, both of \eqref{Sobolev-L-infty} and \eqref{Sobolev-L-infty-weight-L}-\eqref{Sobolev-L-infty-weight-Lb} follow from the Sobolev embedding $L^\infty(\mathbb{R}) \hookrightarrow H^1(\mathbb{R})$ and the facts that
\begin{align*}
 \partial_x \Lambdab(\ub) & \lesssim \Lambdab(\ub), \qquad \qquad \,\,  \partial_x \Lambda(u) \lesssim \Lambda(u), \\
 \partial_x \left( \frac{\Lambdab^{\alpha}(\ub)}{\Lambda^{\beta} (u)} \right) & \lesssim \frac{\Lambdab^{\alpha}(\ub)}{\Lambda^{\beta} (u)}, \quad \partial_x \left( \frac{\Lambda^{\alpha}(u)}{\Lambdab^{\beta} (\ub)}  \right) \lesssim \frac{\Lambda^{\alpha}(u)}{\Lambdab^{\beta} (\ub)}.
\end{align*}
\end{proof}

Here and hereafter, we always use $A\lesssim B$ to denote $A\leq CB$ for some positive constants $C$. And use $A\sim B$ to denote $\frac{A}{C}\leq B\leq CA$.

Based on the $L^\infty$ bound \eqref{Sobolev-L-infty}, following \cite{Wang-Wei}, we find that
\begin{lemma}\label{lemma-energy-formula}
With the bootstrap assumptions \eqref{bt-Eb-Fb} and \eqref{bt-E-F}, we have $g\sim1$ and
\begin{equation}\label{eq}
\begin{split}
\T[\phi_k] (-D u, \TL) & \sim \Lambdab(\ub) (|L \phi_{k}|^2 +|L \phi|^4|\Lb \phi_{k}|^2),\\
\T[\phi_k] (-D \ub, \TL) & \sim \Lambdab(\ub) ( |\Lb \phi|^2|L\phi_{k}|^2+|L\phi|^2|\Lb \phi_{k}|^2),\\
\T[\phi_k] (-D u, \TLb)  & \sim \Lambda (u) (|\Lb \phi|^2|L\phi_{k}|^2+|L\phi|^2|\Lb \phi_{k}|^2),\\
\T[\phi_k] (-D \ub, \TLb) & \sim \Lambda (u) (|\Lb \phi_{k}|^2 +|\Lb \phi|^4|L\phi_{k}|^2),
\end{split}
\end{equation}
providing that $\delta$ is sufficiently small. Consequently,
\begin{equation}\label{eq-dt}
\begin{split}
\T[\phi_k] (-D t, \TL) & \sim \Lambdab(\ub) (|L \phi_{k}|^2 +|L \phi|^2|\Lb \phi_{k}|^2+|\Lb\phi|^2|L\phi_k|^2 +|L \phi|^4|\Lb \phi_{k}|^2),\\
\T[\phi_k] (-D t, \TLb) & \sim \Lambda (u) (|\Lb\phi_k|^2+|\Lb\phi|^2|L\phi_k|^2+|L\phi|^2|\Lb\phi_k|^2  +|\Lb \phi|^4|L\phi_{k}|^2).
\end{split}
\end{equation}
\end{lemma}

\begin{proof}
From Lemma \ref{lem-sobolev}, we know that if $\delta$ is sufficiently small, we have
\begin{equation*}
g=1-L\phi \Lb\phi  \sim 1 \pm \delta M^2 \sim1.
\end{equation*}

For the first line of \eqref{eq}, we calculate that
\begin{align*}
&\T[\phi_k] (D u, \TL) = g^{u\gamma}\partial_{\gamma}\phi_{k}\TL \phi_{k}-\frac{1}{2}g^{\gamma\delta}\partial_{\gamma}\phi_{k}\partial_{\delta}\phi_{k}\delta^{u}_{\alpha} \TL^\alpha \\
&= \Lambdab(\ub) g^{u\gamma}\partial_{\gamma}\phi_{k} \left(  L \phi_{k}+ |L\phi|^2 \Lb \phi_k \right)
-\frac{1}{2} \Lambdab(\ub) |L\phi|^2 \left(- L \phi_k \Lb \phi_k - q(\phi_{k}, \phi_k) \right)\\
&=\Lambdab(\ub) \left(\p^u \phi_k - \frac{1}{g} \p^u \phi \p^\gamma \phi \p_\gamma \phi_k \right) \left(  L \phi_{k}+ |L\phi|^2 \Lb \phi_k \right) + \frac{1}{2} \Lambdab(\ub) |L\phi|^2 L \phi_k \Lb \phi_k  \\
&\quad + \frac{1}{2} \Lambdab(\ub) |L\phi|^2 q (\phi_{k}, \phi_{k}) \\
& = - \frac{1}{2} \Lambdab(\ub) |L \phi_k|^2 + \frac{\Lambdab(\ub)}{2 g} L \phi L \phi_k \p^\gamma \phi \p_\gamma \phi_k +  \frac{1}{2g} \Lambdab(\ub) |L \phi|^2 L \phi \p^\gamma \phi \p_\gamma \phi_k \Lb \phi_k \\
&\quad + \frac{1}{8g} \Lambdab(\ub) |L\phi|^2 \left(|L \phi \Lb \phi_{k}|^2 + |\Lb \phi L \phi_{k}|^2 + 2 L \phi \Lb \phi \Lb \phi_k L \phi_k \right) \\
& =  -\frac{1}{2}  \Lambdab(\ub) |L\phi_{k}|^2 -\frac{1}{8g} \Lambdab(\ub) |L\phi|^4 |\Lb \phi_{k}|^2  + \frac{1}{8g} \Lambdab(\ub)  |L\phi|^2 |\Lb \phi|^2 |L\phi_{k}|^2 \\
&\quad+ \frac{\Lambdab(\ub)}{2 g} L \phi L \phi_k \p^\gamma \phi \p_\gamma \phi_k.
\end{align*}
Then,
\begin{align*}
\T[\phi_k] (-D u, \TL) & = \left( \frac{1}{2}+\frac{L \phi\underline{L}\phi}{4g}\right) \Lambdab(\ub) |L \phi_{k}|^2
 + \frac{1}{8g} \Lambdab(\ub) |L \phi|^4 |\Lb\phi_{k}|^2 \\
&\quad  - \frac{1}{8g} \Lambdab(\ub)  |L \phi|^2 |\Lb \phi|^2 |L \phi_{k}|^2 + \frac{\Lambdab(\ub)}{4 g} |L \phi|^2 L \phi_k \Lb \phi_k \\
& \sim \frac{1}{2} \Lambdab(\ub) |L \phi_{k}|^2 + \frac{1}{8g} \Lambdab(\ub) |L \phi|^4 |\Lb \phi_{k}|^2 - \frac{1}{4g}\Lambdab(\ub) |L\phi|^2 \Lb \phi_{k} L\phi_{k},
\end{align*}
where we use $\left( \frac{1}{2}+\frac{L\phi \Lb \phi}{4g}\right) \Lambdab(\ub) |L \phi_{k}|^2 - \frac{1}{8g} \Lambdab(\ub)  |L\phi|^2 |\Lb \phi|^2 |L \phi_{k}|^2 \sim \frac{1}{2} \Lambdab(\ub) |L \phi_{k}|^2$, if $\delta$ is small enough.
Applying the Cauchy's inequality to the last term
$$
| \frac{1}{4g} \Lambdab(\ub) |L\phi|^2 \Lb \phi_{k} L\phi_{k} | \leq \frac{1}{16g} \Lambdab(\ub) |L \phi|^4 |\Lb \phi_{k}|^2+\frac{1}{4g} \Lambdab(\ub) |L\phi_{k}|^2.
$$
Therefore, it can be absorbed by the first two leading terms $ \frac{1}{2} \Lambdab(\ub) |L\phi_{k}|^2 + \frac{1}{8g} \Lambdab(\ub) |L\phi|^4 |\Lb\phi_{k}|^2$, so that these main terms are roughly $ \frac{1}{4} \Lambdab(\ub) |L\phi_{k}|^2 + \frac{1}{16} \Lambdab(\ub) |L\phi|^4 |\Lb\phi_{k}|^2$, which still exhibit good signs. That is,
$$
\T[\phi_k] (-D u, \TL) \sim \frac{1}{4}  \Lambdab(\ub) \left(|L\phi_{k}|^2+ \frac{1}{4} |L\phi|^4|\underline{L}\phi_{k}|^2 \right).
$$

As for the second one in \eqref{eq},
\begin{align*}
\T[\phi_k] (D \ub, \TL) &= g^{\ub \gamma}\partial_{\gamma}\phi_{k}\tilde{L}\phi_{k}-\frac{1}{2}g^{\gamma\delta}\partial_{\gamma}\phi_{k}\partial_{\delta}\phi_{k}\delta^{\ub}_{\alpha} \tilde{L}^\alpha \\
&= \Lambdab(\ub) g^{\ub \gamma}\partial_{\gamma}\phi_{k} \left(  L \phi_{k}+ |L\phi|^2 \Lb \phi_k \right)
-\frac{1}{2} \Lambdab(\ub)  \left(- L \phi_k \Lb \phi_k - q(\phi_{k}, \phi_k) \right)\\
&=\Lambdab(\ub) \left(\p^{\ub} \phi_k - \frac{1}{g} \p^{\ub} \phi \p^\gamma \phi \p_\gamma \phi_k \right) \left(  L \phi_{k}+ |L\phi|^2 \Lb \phi_k \right) + \frac{1}{2} \Lambdab(\ub) L \phi_k \Lb \phi_k  \\
&\quad + \frac{1}{2} \Lambdab(\ub) q (\phi_{k}, \phi_{k}) \\
& = - \frac{1}{2} \Lambdab(\ub) |L \phi|^2 |\Lb \phi_k|^2 + \frac{\Lambdab(\ub)}{2 g} \Lb \phi \p^\gamma \phi \p_\gamma \phi_k L \phi_k  +  \frac{1}{2g} \Lambdab(\ub)  \Lb \phi \p^\gamma \phi \p_\gamma \phi_k |L \phi|^2 \Lb \phi_k \\
&\quad + \frac{1}{8g} \Lambdab(\ub)  \left(|L \phi \Lb \phi_{k}|^2 + |\Lb \phi L \phi_{k}|^2 + 2 L \phi \Lb \phi \Lb \phi_k L \phi_k \right).
\end{align*}
It follows by straightforward calculations that
\begin{align*}
\T[\phi_k] (-D \ub, \TL) &= \left( \frac{1}{2} - \frac{1}{8g}\right) \Lambdab(\ub) |L \phi |^2 |\Lb \phi_k|^2 + \frac{1}{8g} \Lambdab(\ub) |\Lb \phi |^2 |L \phi_k|^2  + Er_3, \\
Er_3 &=   - \frac{Q(\phi, \phi_k)}{2g} L\phi \Lb \phi L\phi \Lb \phi_k.
\end{align*}
We can check that
$|Er_3 | \lesssim  \delta M^2  |L\phi|^2 |\Lb \phi_{k}|^2+\delta M^2 |\Lb \phi|^2|L\phi_{k}|^2$. Thus,
$$
\T[\phi_k]  (-D \ub, \tilde{L}) \sim \frac{1}{8} \Lambdab(\ub) (3|L \phi |^2 |\Lb \phi_k|^2 +  |\Lb \phi |^2 |L \phi_k|^2).
$$
The third and fourth lines in \eqref{eq} can be proved in an analogous way.

\end{proof}

We also need the higher order equations whose double null structure is crucial for our energy estimates.
\begin{lemma}\label{lemma-high-order-eq}
The higher order $\phi_k$ satisfies the following equation
\begin{equation}
\begin{split}
& \Box_{g(\p\phi)}\phi_k   = (1+Q)^{-1} g^{\mu \nu} \partial_{\nu}Q \partial_{\mu}\partial^k \phi \\
+ & \sum_{ \substack{|i_1|+\cdots +|i_j| \leq |k|,\\ 3 \leq j \leq 2|k|+3} } F_{i_1, \cdots, i_j}(Q)Q(\partial^{i_1} \phi, Q( \partial^{i_2}\phi, \partial^{i_3} \phi) )\\
&\quad \quad \quad \quad \quad \quad  \times Q(\partial^{i_4} \phi,\partial^{i_5} \phi) \cdots Q( \partial^{i_{j-1}}\phi, \partial^{i_j} \phi),
\end{split}
\end{equation}
where $F_{i_1, \cdots i_j}(Q)$ is a fractional function on $Q = Q(\phi,\phi)=\eta^{\alpha\beta}\p_{\alpha} \phi \p_{\beta} \phi$ and $j$ is odd, and if $j = 3$, then this is to be interpreted as the factor $Q(\partial^{i_4} \phi,\partial^{i_5} \phi)$ being absent.
\end{lemma}
The proof of Lemma \ref{lemma-high-order-eq} is done by performing $\partial=(\partial_t,\partial_x)$ on the relativistic string equation \eqref{2.7} for $|k|$ times. For more details, please refer to \cite{Wang-Wei}.

Based on the above preparations, we will now continue with the energy estimates. We intend to prove the following proposition.
\begin{proposition}\label{pro}
Let $N\geq 4$, under the bootstrap assumptions \eqref{bt-Eb-Fb} and \eqref{bt-E-F}, there exists some positive constant $C_1$ such that
\begin{subequations}
\begin{align}
\Eb^2_{N+1} (\ub,t)+ \Fb^2_{N+1}(\ub, t) & \leq I_{N+1}^2 + C_1\delta M^4,\label{energy1}\\
E^2_{N+1}(u,t) + F^2_{N+1}(u, t) & \leq  \delta^2 I^2_{N+1} +C_1 \delta^3 M^6,\label{energy2}
\end{align}
\end{subequations}
for all $t \geq 0$ and $u, \, \ub \in \mathbb{R}$. Here $I_{N+1}$ is a constant depending only on the initial data (up to $N+1$ order derivatives).
\end{proposition}

Once Proposition \ref{pro} is proved, we can improve the bootstrap assumptions \eqref{bt-Eb-Fb}-\eqref{bt-E-F} as
\begin{align}
& \Eb^2_{N+1}(t )+ \sup_{\ub \in \mathbb{R}} \Fb^2_{N+1}(\ub, t) \leq \frac{M^2}{2}, \label{bt-Eb-Fb-improve}\\
& E^2_{N+1}(t) + \sup_{u \in \mathbb{R}} F^2_{N+1}(u, t) \leq \frac{\delta^2 M^2}{2}, \label{bt-E-F-improve}
\end{align}
and close the bootstrap arguments as follow.
Let $[0, T^\ast]$ be the largest time interval so that  the bootstrap assumptions \eqref{bt-Eb-Fb} and \eqref{bt-E-F} hold true. By the local well poseness for wave equations, we know that $T^\ast >0$. The improvements \eqref{bt-Eb-Fb-improve} and \eqref{bt-E-F-improve} indicate that the bootstrap assumptions \eqref{bt-Eb-Fb}, \eqref{bt-E-F} can be extended to a larger time interval, thus contradicting the maximality of $T^\ast$. And hence we must have $T^\ast = +\infty$.

To derive the improvement, we choose the constant $M$ large enough so that $I_{N+1}^2 \leq \frac{M^2}{4}$ (note that, $M$ depends only on the initial data, in particular, it is independent of $\delta$), and $\delta$ small enough so that $C_1\delta M^4\leq \frac{M^2}{4}$. Thus, \eqref{bt-Eb-Fb-improve} and \eqref{bt-E-F-improve} hold true.

\subsection{Proof of Proposition \ref{pro}}
In this subsection, we will prove Proposition \ref{pro}. As we have explained before, we need to control the $\Eb, \, \Fb$ energies first.
\subsubsection{Estimates for the energies $\Eb_{N+1}$ and $\Fb_{N+1}$.}

Letting $\xi = \TLb$, $\varphi =\phi_k$ in \eqref{energy-id-1}, we have the energy identity
\begin{equation}\label{eq-energy-TLb}
\begin{split}
& \int_{\Sigma^{-}_{t, \ub}} \T[\phi_k] (-D\tau, \TLb) \sqrt{g} \di x+ \int_{\Cb^t_{\ub}} \T[\phi_k] (-D \ub, \TLb) \sqrt{g} \di u \\
={} & \int_{\Sigma^{-}_{0, \ub}} \T[\phi_k] (-D\tau, \TLb)  \sqrt{g} \di x -\iint_{\D^{-}_{t, \ub}}\p_{\alpha}(\sqrt{g}P^{\alpha} [\phi_k,\TLb]) \di \tau\di x.
\end{split}
\end{equation}
Noting \eqref{eq-div-P}, which tells that
\begin{align*}
\frac{1}{\sqrt{g}}\p_{\alpha}(\sqrt{g}P^{\alpha}[\phi_k, \TLb])= \Box_{g(\p\phi)}\phi_k \cdot \TLb\phi_k + \T^{\alpha}_{\beta} [\phi_k] \cdot \p_{\alpha}\TLb^{\beta} - \frac{1}{2\sqrt{g}}\TLb(\sqrt{g}g^{\gamma\rho}) \p_{\gamma}\phi_k \p_{\rho}\phi_k,
\end{align*}
we will bound the nonlinear error terms on the right hand side in what follows.

{\bf The source term:} $\Box_{g(\p\phi)}\phi_k \cdot \TLb \phi_k$.

According to Lemma \ref{lemma-high-order-eq}, $\Box_{g(\p\phi)}\phi_k$ satisfies the cubic double null condition in Minkowski geometry. We can bound it as
\begin{align*}
|\Box_{g(\p\phi)}\phi_k| & \lesssim \sum_{\substack{i, j \leq p \leq k, \\ p+i+j \leq k}} |\Lb \phi_p L^2 \phi_i \Lb \phi_j| + |\Lb \phi_p L \Lb \phi_i L \phi_j| + |L \phi_p \Lb^2 \phi_i L \phi_j| + |L \phi_p \Lb L \phi_i \Lb \phi_j| \\
&+ \sum_{\substack{i, j \leq q <k, \\ q+i+j \leq k}} | L^2 \phi_q \Lb \phi_i \Lb \phi_j| + |L \Lb \phi_q \Lb \phi_i L \phi_j| + |\Lb^2 \phi_q L \phi_i L \phi_j|.
\end{align*}
Substituting the $L^\infty$ estimates of the lower order derivative terms, we have
\begin{align*}
|\Box_{g(\p\phi)}\phi_k| & \lesssim \sum_{\substack{i, j \leq p \leq k, \\ p+i+j \leq k}}  \frac{\delta M^2}{\Lambdab^{\frac{1}{2}} (\ub) \Lambda^{\frac{1}{2}} (u)}  |\Lb \phi_p| + |L \phi_p \Lb^2 \phi_i L \phi_j| + \frac{M^2}{\Lambda (u)}   |L \phi_p| \\
&+ \sum_{\substack{i, j \leq q <k, \\ q+i+j \leq k}} \frac{M^2}{\Lambda (u)}  | L^2 \phi_q| +  \frac{\delta M^2}{\Lambdab^{\frac{1}{2}} (\ub) \Lambda^{\frac{1}{2}} (u)} |\Lb L \phi_q| + \frac{\delta^2 M^2}{\Lambdab (\ub) }  |\Lb^2 \phi_q|.
\end{align*}
Here we keep the term $|L \phi_p \Lb^2 \phi_i L \phi_j|$ unchanged since it needs a different treatment to gain sufficient decay. This can be found in \eqref{source-Lb-special}.

Specifically, we split the source term into two parts: $\Box_{g(\p\phi)}\phi_k \cdot \TLb \phi_k = \Box_{g(\p\phi)}\phi_k \cdot \Lambda (u) \Lb \phi_k +  \Box_{g(\p\phi)}\phi_k \cdot \Lambda (u) (\Lb \phi)^2 L \phi_k$. The first part $\Box_{g(\p\phi)}\phi_k \cdot \Lambda (u) \Lb \phi_k$, which is also the dominating one, can be estimated as follows. Among all the calculations below, we should always remind ourselves that $p \leq k$ and $q +1\leq k$. Here we have,
\begin{align*}
& \iint_{\D^{-}_{t, \ub}} \Lambda  (u) |\Lb \phi_k| \left(\frac{\delta M^2}{\Lambdab^{\frac{1}{2}} (\ub) \Lambda^{\frac{1}{2}} (u)} ( |\Lb \phi_p|+  |\Lb L \phi_q| ) + \frac{\delta^2 M^2}{\Lambdab (\ub) }  |\Lb^2 \phi_q|  \right)  \sqrt{g} \di \tau \di x \\
\lesssim {} &   \int^{\ub}_{-\infty} \frac{\delta M^2}{\Lambdab^{\frac{1}{2}}(\ub^\prime)} \di \ub^\prime \int_{\Cb^t_{\ub^\prime}} \Lambda (u) \sum_{j \leq k} |\Lb \phi_j|^2 \sqrt{g} \di u \lesssim \int^{\ub}_{-\infty} \frac{\delta M^2}{\Lambdab^{\frac{1}{2}}(\ub^\prime)} \Fb_{k+1}^2 (\ub^\prime, t) \di \ub^\prime \\
\lesssim {}& \delta M^4.
\end{align*}
For the term involving $|L \phi_p \Lb^2 \phi_i L \phi_j|$, we follow \cite{Luli-Yang-Yu} to derive
\begin{align}
& \iint_{\D^{-}_{t, \ub}}\Lambda (u) | \Lb \phi_k | |L \phi_p \Lb^2 \phi_i L \phi_j| \sqrt{g} \di \tau \di x \nnb \\
\lesssim {} & \iint_{\D^{-}_{t, \ub}}\Lambda (u) | \Lb \phi_k | |L \phi_p \Lb^2 \phi_i| \frac{\delta M}{\Lambdab^{\frac{1}{2}}(\ub^\prime)} \sqrt{g} \di \tau \di x \nnb \\
\lesssim {} &\delta M \left(\iint_{\D^{-}_{t, \ub}} \frac{\Lambda (u)}{\Lambdab (\ub^\prime)} |\Lb \phi_k|^2 \sqrt{g} \di \tau \di x \right)^{\frac{1}{2}} \nnb \\
& \qquad \qquad \quad \cdot \sup_{\tau} \left(\int_{\Sigma^{-}_{\tau, \ub}} \Lambdab(\ub^\prime)  |L \phi_p|^2 \sqrt{g} \di x \right)^{\frac{1}{2}} \left(\int_0^t \Big\| \frac{\Lambda^{\frac{1}{2}}(u)}{\Lambdab^{\frac{1}{2}} (\ub^\prime)} |\Lb^2 \phi_i|  \Big\|^2_{L^\infty (\Sigma^{-}_{\tau, \ub}) } \di \tau \right)^{\frac{1}{2}} \nnb \\
\lesssim {} & \delta M \sup_\tau E^2_{p+1} (\tau) \left(\iint_{\D^{-}_{t, \ub}} \frac{\Lambda (u)}{\Lambdab (\ub^\prime)} |\Lb \phi_k|^2 \sqrt{g} \di \tau \di x  \right)^{\frac{1}{2}} \sum_{l \leq 1} \left(\iint_{\D^{-}_{t, \ub} } \frac{\Lambda(u)}{\Lambdab (\ub^\prime)} |\Lb^2 \p^l_x \phi_i|^2 \di \tau \di x  \right)^{\frac{1}{2}} \nnb \\
 \lesssim {} &  \delta^2 M^2 \sum_{l \leq N} \int^{\ub}_{-\infty} \frac{1}{\Lambdab (\ub^\prime)} \di \ub^\prime \int_{\Cb^t_{\ub^\prime}} \Lambda(u) |\Lb \phi_l|^2 \di u \lesssim \int^{\ub}_{-\infty} \frac{ \delta^2 M^2}{\Lambdab (\ub^\prime)}  \Fb_{N+1}^2 (\ub^\prime, t) \di \ub^\prime \nnb \\
\lesssim {}& \delta^2 M^4, \label{source-Lb-special}
\end{align}
where in the third inequality, we have used the weighted Sobolev inequality \eqref{Sobolev-L-infty-weight-Lb}\footnote{Since \eqref{Sobolev-L-infty-weight-Lb} and \eqref{Sobolev-L-infty-weight-L} are frequently used in our estimates, we will not always point it out.}.
And then we are left with the term
\begin{align}
& \iint_{\D^{-}_{t, \ub}} \Lambda  (u) |\Lb \phi_k|  \frac{M^2}{\Lambda (u)} \left(  |L \phi_p| + | L^2 \phi_q|  \right) \sqrt{g} \di \tau \di x \lesssim  \iint_{\D^{-}_{t, \ub}} M^2 |\Lb \phi_k| \sum_{i \leq k} |L \phi_i|  \sqrt{g} \di \tau \di x \nnb \\
\lesssim {} & M^2 \left(\iint_{\D^{-}_{t, \ub}} \frac{\Lambda (u)}{\Lambdab (\ub)} |\Lb \phi_k|^2 \sqrt{g} \di \tau \di x \right)^{\frac{1}{2}} \sum_{i \leq k} \left(\iint_{\D^{-}_{t, \ub}} \frac{\Lambdab (\ub)}{\Lambda (u)} |L \phi_i|^2 \sqrt{g} \di \tau \di x \right)^{\frac{1}{2}}\nnb \\
\lesssim {} & \left(\int^{\ub}_{-\infty} \frac{M}{\Lambdab (\ub^\prime)} \di \ub^\prime \int_{\Cb^t_{\ub^\prime}} \Lambda (u) |\Lb \phi_k|^2 \sqrt{g} \di u \right)^{\frac{1}{2}} \left(\int^{+\infty}_{-\infty} \frac{M}{\Lambda (u)} \di u \int_{C^t_{u}} \Lambdab (\ub^\prime) \sum_{i \leq k} |L \phi_i|^2 \sqrt{g} \di \ub^\prime \right)^{\frac{1}{2}}  \nnb \\
\lesssim {} & \left(\int^{\ub}_{-\infty} \frac{M}{\Lambdab (\ub^\prime)} \Fb_{k+1}^2 (\ub^\prime, t) \di \ub^\prime \right)^{\frac{1}{2}} \left(\int^{+\infty}_{-\infty} \frac{M}{\Lambda (u)} F_{k+1}^2(u, t) \di u \right)^{\frac{1}{2}} \nnb \\
\lesssim {}& \delta M^4. \label{terms-L-Lb}
\end{align}
In fact, \eqref{terms-L-Lb} represents a typical estimate for terms like $\iint  |\Lb \phi_i|  |L \phi_j|  \di \tau \di x$. And we will always come to this point whenever we are encountered with this situation.

Next, we come to the second part $\Box_{g(\p\phi)}\phi_k \cdot \Lambda (u) (\Lb \phi)^2 L \phi_k$.
\begin{align*}
& \iint_{\D^{-}_{t, \ub}} \Lambda  (u) |\Lb \phi|^2 |L \phi_k| \left(\frac{\delta M^2}{\Lambdab^{\frac{1}{2}} (\ub) \Lambda^{\frac{1}{2}} (u)} ( |\Lb \phi_p|+  |\Lb L \phi_q| ) + \frac{\delta^2 M^2}{\Lambdab (\ub) }  |\Lb^2 \phi_q|  \right) \sqrt{g} \di \tau \di x  \\
\lesssim & \iint_{\D^{-}_{t, \ub}} M^2 |L \phi_k| \left(\frac{\delta M^2}{\Lambdab^{\frac{1}{2}} (\ub) \Lambda^{\frac{1}{2}} (u)} ( |\Lb \phi_p|+  |\Lb L \phi_q| ) + \frac{\delta^2 M^2}{\Lambdab (\ub) }  |\Lb^2 \phi_q|  \right)  \sqrt{g} \di \tau \di x \\
\lesssim & \iint_{\D^{-}_{t, \ub}} \delta M^4 |L \phi_k|  \sum_{i \leq k} |\Lb \phi_i| \sqrt{g} \di \tau \di x  \quad \text{by \eqref{terms-L-Lb}}   \\
\lesssim {}& \delta^2 M^6.
\end{align*}
And
\begin{align*}
& \iint_{\D^{-}_{t, \ub}} \Lambda  (u) |\Lb \phi|^2 |L \phi_k| \left( \frac{M^2}{\Lambda (u)} \left(  |L \phi_p| + | L^2 \phi_q|  \right) +  |L \phi_p \Lb^2 \phi_i L \phi_j| \right) \sqrt{g} \di \tau \di x \\
\lesssim {} & \iint_{\D^{-}_{t, \ub}} M^2 |L \phi_k| \left( \frac{M^2}{\Lambda (u)} \left(  |L \phi_p| + | L^2 \phi_q|  \right) +   \frac{\delta M^2}{\Lambdab^{\frac{1}{2}} (\ub) \Lambda^{\frac{1}{2}} (u)} |L \phi_p| \right) \sqrt{g} \di \tau \di x \\
\lesssim {} & M^4 \sum_{i \leq k} \iint_{\D_{t}} \frac{\Lambdab (\ub)}{\Lambda^{\frac{1}{2}} (u)} |L \phi_k| |L \phi_i| \sqrt{g} \di \tau \di x \lesssim \delta^2 M^6.
\end{align*}

{\bf The deformation tensor:} $\T^{\alpha}_{\beta} [\phi_k] \cdot \p_{\alpha}\TLb^{\beta} - \frac{1}{2\sqrt{g}}\TLb(\sqrt{g}g^{\gamma\rho}) \p_{\gamma}\phi_k \p_{\rho}\phi_k$.

For $\T^{\alpha}_{\beta} [\phi_k] \cdot \p_{\alpha}\TLb^{\beta}$, in views of Lemma \ref{lem-deformation},
\begin{align}
|\T^{\alpha}_{\beta} [\phi_k]  \p_{\alpha} \TLb^{\beta}| \lesssim {}& \left( |L \phi_k|^2 + |L \phi \Lb \phi| |L \phi_k|^2 + |L \phi|^2 |L \phi_k \Lb \phi_k|  \right)  |\Lambda^\prime (u)| |\Lb \phi|^2  \nnb \\
& + \left( |L \phi_k|^2 + |L \phi \Lb \phi| |L \phi_k|^2 + |L \phi|^2 |L \phi_k \Lb \phi_k|  \right)  \Lambda (u) \Lb^2 \phi \Lb \phi  \nnb \\
& + \left( |\Lb \phi|^2 |L \phi_k|^2 +  |L \phi|^2 |\Lb \phi_k|^2  \right) \cdot  \Lambda (u) L \Lb  \phi \Lb \phi. \label{deformation-Lb}
\end{align}
In the first line of \eqref{deformation-Lb}, there is the leading term in $ \left( |L \phi_k|^2 + |L \phi \Lb \phi| |L \phi_k|^2  \right)  |\Lambda^\prime (u)| |\Lb \phi|^2$ as follow,
\begin{align}
& \iint_{\D^{-}_{t, \ub}} |L \phi_k|^2 |\Lambda^\prime (u)| |\Lb \phi|^2 \sqrt{g} \di \tau \di x  \leq  \iint_{\D^{-}_{t, \ub}} \Lambda (u) |L \phi_k|^2  |\Lb \phi|^2 \sqrt{g} \di \tau \di x \nnb \\
\lesssim &{} \sup_\tau \int_{\Sigma^{-}_{\tau, \ub}} \Lambdab (\ub^\prime) |L \phi_k|^2 \sqrt{g} \di x  \int_0^t \Big\|\frac{\Lambda^{\frac{1}{2}}(u)}{\Lambdab^{\frac{1}{2}}(\ub^\prime)} \Lb \phi\Big\|^2_{L^\infty (\Sigma^{-}_{\tau, \ub})} \di \tau \nnb \\
\lesssim &{} \sup_\tau E^2_{k+1} (\tau) \iint_{\D^{-}_{t, \ub}} \frac{\Lambda(u)}{\Lambdab(\ub^\prime)} (|\Lb \phi|^2 + |\Lb \p_x \phi|^2) \di \tau \di x \nnb \\
\lesssim &{} \delta^2 M^4, \label{esti-leading-Lb}
\end{align}
and the remaining term
\begin{align*}
& \iint_{\D^{-}_{t, \ub}} |L \phi|^2 |L \phi_k \Lb \phi_k| |\Lambda^\prime (u)| |\Lb \phi|^2  \sqrt{g} \di \tau \di x \\
\lesssim {} & \iint_{\D^{-}_{t, \ub}} \frac{\delta^2 M^4}{\Lambdab(\ub^\prime)} |L \phi_k \Lb \phi_k|  \sqrt{g} \di \tau \di x \lesssim \delta^3 M^6
\end{align*}
 can be treated in an analogous fashion as \eqref{terms-L-Lb}.
The second line of \eqref{deformation-Lb} is similar to the first line, hence it suffices to check the following leading term,
\begin{align}
& \iint_{\D^{-}_{t, \ub}} |L \phi_k|^2 \Lambda (u) |\Lb^2 \phi \Lb \phi|  \sqrt{g} \di \tau \di x \nnb \\
\lesssim &{} \sup_\tau \int_{\Sigma^{-}_{\tau, \ub}} \Lambdab (\ub^\prime) |L \phi_k|^2  \sqrt{g} \di x \int_0^t \Big\|\frac{\Lambda^{\frac{1}{2}}(u)}{\Lambdab^{\frac{1}{2}}(\ub^\prime)} \Lb \phi\Big\|_{L^\infty (\Sigma^{-}_{\tau, \ub})} \Big\|\frac{\Lambda^{\frac{1}{2}}(u)}{\Lambdab^{\frac{1}{2}}(\ub^\prime)} \Lb^2 \phi\Big\|_{L^\infty (\Sigma^{-}_{\tau, \ub})} \di \tau \nnb \\
\lesssim &{} \delta^2 M^2 \int_0^t \left( \Big\|\frac{\Lambda^{\frac{1}{2}}(u)}{\Lambdab^{\frac{1}{2}}(\ub^\prime)} \Lb \phi\Big\|^2_{L^\infty (\Sigma^{-}_{\tau, \ub})} + \Big\|\frac{\Lambda^{\frac{1}{2}}(u)}{\Lambdab^{\frac{1}{2}}(\ub^\prime)} \Lb^2 \phi\Big\|^2_{L^\infty (\Sigma^{-}_{\tau, \ub})}\right) \di \tau \nnb \\
\lesssim &{} \delta^2 M^2 \sum_{i \leq 1} \iint_{\D^{-}_{t, \ub}} \frac{\Lambda(u)}{\Lambdab(\ub^\prime)} (|\Lb \p^i_x \phi|^2 + |\Lb^2 \p^i_x \phi|^2) \di \tau \di x \lesssim  \delta^2 M^4. \label{esti-crucial-Lb}
\end{align}
For the third line of \eqref{deformation-Lb},
\begin{align*}
& \iint_{\D^{-}_{t, \ub}} \left( |\Lb \phi|^2 |L \phi_k|^2 +  |L \phi|^2 |\Lb \phi_k|^2  \right)   \Lambda (u)  |L \Lb  \phi \Lb \phi| \sqrt{g} \di \tau \di x \\
\lesssim {} & \iint_{\D^{-}_{t, \ub}} \left( \frac{\delta M^4}{\Lambda^{\frac{1}{2}}(u)}  |L \phi_k|^2 + \frac{\delta^3 M^4}{\Lambdab(\ub^\prime)} \Lambda (u)  |\Lb \phi_k|^2 \right) \sqrt{g} \di \tau \di x  \lesssim  \delta^3 M^6. \\
\end{align*}

For $\frac{1}{2\sqrt{g}}\TLb(\sqrt{g} g^{\gamma\rho}) \p_{\gamma}\phi_k \p_{\rho}\phi_k$, we calculate straightforwardly,
\begin{align}
& \frac{1}{2\sqrt{g}}\TLb(\sqrt{g}g^{\gamma\rho}) \p_{\gamma}\phi_k \p_{\rho}\phi_k = \frac{\TLb Q(\phi, \phi)}{2 g}  g^{\gamma\rho} \p_{\gamma}\phi_k \p_{\rho}\phi_k + \TLb  g^{\gamma\rho} \p_{\gamma}\phi_k \p_{\rho}\phi_k \nnb \\
= {} & \frac{\TLb \p_\mu \phi \p^\mu \phi}{g} g^{\gamma\rho} \p_{\gamma}\phi_k \p_{\rho}\phi_k -  \frac{2}{g} \TLb \p^\gamma \phi \p^\rho \phi \p_{\gamma}\phi_k \p_{\rho}\phi_k + \frac{2 \TLb \p_\mu \phi \p^\mu \phi}{g^2} \p^\gamma \phi \p^\rho \phi \p_{\gamma}\phi_k \p_{\rho}\phi_k \nnb \\
={}& \frac{1}{2g} \left(\TLb \Lb \phi L \phi + \TLb L \phi \Lb \phi \right) \Lb \phi_k L \phi_k -  \frac{1}{2g} ( \TLb \Lb \phi L \phi \Lb \phi_k L \phi_k + \TLb L \phi \Lb \phi L \phi_k \Lb \phi_k ) \nnb \\
&-  \frac{1}{2g} ( \TLb \Lb \phi \Lb \phi L \phi_k L \phi_k + \TLb L \phi L \phi \Lb \phi_k \Lb \phi_k) + \frac{ \TLb \p_\mu \phi \p^\mu \phi}{g^2} \p^\gamma \phi \p^\rho \phi \p_{\gamma}\phi_k \p_{\rho}\phi_k \nnb \\
={}& -  \frac{1}{2g} ( \TLb \Lb \phi \Lb \phi |L \phi_k|^2 + \TLb L \phi L \phi |\Lb \phi_k|^2) + \frac{ \TLb \p_\mu \phi \p^\mu \phi}{g^2} \p^\gamma \phi \p^\rho \phi \p_{\gamma}\phi_k \p_{\rho}\phi_k, \label{deform-error-Lb-2}
\end{align}
where we notice that the third line above cancels (although this cancellation is irrelevant at all, because each term in the third line enjoys good estimate).

In the two portions of $\TLb = \Lambda (u) \Lb + \Lambda (u) |\Lb \phi|^2 L$, the contribution of $\Lambda (u) \Lb$ is always larger than $ \Lambda (u) |\Lb \phi|^2 L$. This can be verified from the following comparisons
$$|\Lb \p \phi| \lesssim \frac{M}{ \Lambda^{\frac{1}{2}} (u)}, \quad |\Lb \phi|^2 |L \p \phi| \lesssim \frac{\delta M^3}{\Lambdab^{\frac{1}{2}} (\ub) \Lambda(u)}.$$
Consequently, it suffices to check the following terms.
\begin{align*}
& \iint_{\D^{-}_{t, \ub}} \Lambda (u) |L\Lb \phi L\phi| |\Lb \phi_k|^2 \sqrt{g} \di \tau \di x  \\
\lesssim {}&  \iint_{\D^{-}_{t, \ub}} \frac{\delta^2 M^2}{\Lambdab(\ub^\prime)} \Lambda (u) |\Lb \phi_k|^2 \sqrt{g} \di \tau \di x   \lesssim  \delta^2 M^4.
\end{align*}
And the bound $\iint_{\D^{-}_{t, \ub}} \Lambda (u) |\Lb^2 \phi \Lb \phi|  |L \phi_k|^2 \sqrt{g} \di \tau \di x  \lesssim \delta^2 M^4$ has been shown in \eqref{esti-crucial-Lb}.
Eventually, among all the terms in the final result \eqref{deform-error-Lb-2}, the last one $\frac{ \TLb \p_\mu \phi \p^\mu \phi}{g^2} \p^\gamma \phi \p^\rho \phi \p_{\gamma}\phi_k \p_{\rho}\phi_k$ is obvious a lower order term. Anyway, it is straightforward to find that $$\iint_{\D^{-}_{t, \ub}} \frac{ |\TLb \p_\mu \phi \p^\mu \phi |}{g^2} | \p^\gamma \phi \p^\rho \phi \p_{\gamma}\phi_k \p_{\rho}\phi_k|  \sqrt{g} \di \tau \di x \lesssim \delta^3 M^6.$$

Combining all the above calculations and summing over all $k\leq N$, we achieve
\bel{energy-bound-Eb-Fb}
\Eb^2_{N+1}(\ub,t)+ \Fb^2_{N+1} (\ub, t) \lesssim \Eb^2_{N+1} (\ub, 0) + \delta M^4  \lesssim I_{N+1}^2 + \delta M^4.
\ee
As a result of \eqref{energy-bound-Eb-Fb}, we can update the following $L^\infty$ estimates,
\bel{improved-L-infty-Lb}
|\Lb \phi_k| \lesssim \frac{I_{N+1} + \sqrt{\delta} M^2}{\Lambda^{\frac{1}{2}} (u)}, \quad k \leq N-1.
\ee
Both the improvements \eqref{energy-bound-Eb-Fb} and \eqref{improved-L-infty-Lb} will be crucial in estimating the $E, \, F$, energies which we will present in the next subsection.

\subsubsection{Estimates for energies $E^2_{N+1}$ and $F^2_{N+1}$.}
Letting $\xi = \TL, \, \varphi = \phi_k$ in \eqref{energy-id-2}, we have
\begin{equation}\label{eq-energy-TL}
\begin{split}
& \int_{\Sigma^{+}_{t, u}} \T[\phi_k] (-D\tau, \TL) \sqrt{g} \di x+ \int_{C^t_{u}} \T[\phi_k] (-D u, \TL) \sqrt{g} \di u \\
&= \int_{\Sigma^{+}_{0, u}} \T[\phi_k] (-D\tau, \TL)\sqrt{g} \di x-\iint_{\D^{+}_{t, u}}\p_{\alpha}(\sqrt{g}P^{\alpha} [\phi_k,\TL]) \di \tau \di x.
\end{split}
\end{equation}
And in views of \eqref{eq-div-P}, the nonlinear error terms take the form of
\begin{align*}
\frac{1}{\sqrt{g}}\p_{\alpha}(\sqrt{g}P^{\alpha}[\phi_k, \TL])= \Box_{g(\p\phi)}\phi_k \cdot \TL \phi_k + \T^{\alpha}_{\beta} (\phi_k) \cdot \p_{\alpha}\TL^{\beta} - \frac{1}{2\sqrt{g}}\TL (\sqrt{g}g^{\gamma \rho}) \p_{\gamma}\phi_k \p_{\rho}\phi_k.
\end{align*}

{\bf The source term:} $\Box_{g(\p\phi)}\phi_k \cdot \TL \phi_k$.

As what has been explained before, there is
\begin{align*}
|\Box_{g(\p\phi)}\phi_k| & \lesssim \sum_{\substack{i, j \leq p \leq k, \\ p+i+j \leq k}} |\Lb \phi_p L^2 \phi_i \Lb \phi_j| + |\Lb \phi_p L \Lb \phi_i L \phi_j| + |L \phi_p \Lb^2 \phi_i L \phi_j| + |L \phi_p L \Lb \phi_i \Lb \phi_j| \\
&+ \sum_{\substack{i, j \leq q <k, \\ q+i+j \leq k}} | L^2 \phi_q \Lb \phi_i \Lb \phi_j| + |L \Lb \phi_q \Lb \phi_i L \phi_j| + |\Lb^2 \phi_q L \phi_i L \phi_j|.
\end{align*}
Again we substitute the $L^\infty$ estimates for the lower order derivative terms to obtain
\begin{align*}
|\Box_{g(\p\phi)}\phi_k| & \lesssim \sum_{i, j \leq p \leq k, \, p+i+j \leq k} |\Lb \phi_p L^2 \phi_i \Lb \phi_j| + \frac{\delta^2 M^2}{\Lambdab (\ub) }  |\Lb \phi_p| +  \frac{\delta M^2}{\Lambdab^{\frac{1}{2}} (\ub) \Lambda^{\frac{1}{2}} (u)}  |L \phi_p| \\
&+ \sum_{i, j \leq q <k, \, q+i+j \leq k} \frac{I^2_{N+1} + \delta M^4}{\Lambda (u)}  | L^2 \phi_q| +  \frac{\delta M^2}{\Lambdab^{\frac{1}{2}} (\ub) \Lambda^{\frac{1}{2}} (u)} |L \Lb \phi_q| + \frac{\delta^2 M^2}{\Lambdab (\ub) }  |\Lb^2 \phi_q|.
\end{align*}
In particular, the improvement \eqref{improved-L-infty-Lb} is needed in bounding $| L^2 \phi_q \Lb \phi_i \Lb \phi_j|$.
Besides, different from the previous case, here we keep the form of the term $|\Lb \phi_p L^2 \phi_i \Lb \phi_j|$ so that it can be estimated in the following way.

We also split the source term into two parts: $\Box_{g(\p\phi)}\phi_k \cdot \TL \phi_k = \Box_{g(\p\phi)}\phi_k \cdot \Lambdab (\ub) L \phi_k +  \Box_{g(\p\phi)}\phi_k \cdot \Lambdab (\ub) (L \phi)^2 \Lb \phi_k$. And the first part $\Box_{g(\p\phi)}\phi_k \cdot \Lambdab (\ub) L \phi_k$ is the primary term, for which we will begin with the following estimates,
\begin{align*}
& \iint_{\D^{+}_{t, u}}\Lambdab (\ub) |L \phi_k | |\Lb \phi_p L^2 \phi_i \Lb \phi_j| \sqrt{g} \di \tau \di x  \\
\lesssim {} & \iint_{\D^{+}_{t, u}}\Lambdab (\ub) |L \phi_k | |\Lb \phi_p L^2 \phi_i | \frac{I_{N+1} + \sqrt{\delta} M^2}{\Lambda^{\frac{1}{2}} (u^\prime) } \sqrt{g} \di \tau \di x  \\
\lesssim {}& (I_{N+1} + \sqrt{\delta} M^2) \left(\iint_{\D^{+}_{t, u}} \frac{\Lambdab (\ub)}{\Lambda (u^\prime)} |L \phi_k|^2 \sqrt{g} \di \tau \di x  \right)^{\frac{1}{2}} \\
& \qquad  \cdot \sup_\tau \left(\int_{\Sigma^{+}_{\tau, u}} \Lambda(u^\prime)  |\Lb \phi_p|^2 \sqrt{g} \di x  \right)^{\frac{1}{2}} \left(\int_0^t \Big\| \frac{\Lambdab^{\frac{1}{2}}(\ub)}{\Lambda^{\frac{1}{2}} (u^\prime)} |L^2 \phi_i|  \Big\|^2_{L^\infty (\Sigma^{+}_{\tau, u}) } \di \tau \right)^{\frac{1}{2}} \\
\lesssim {} &  (I_{N+1} + \sqrt{\delta} M^2) \sup_\tau \Eb^2_{p+1} (t) \left(\iint_{\D^{+}_{t, u}} \frac{\Lambdab (\ub)}{\Lambda(u^\prime)} |L \phi_k|^2 \sqrt{g} \di \tau \di x  \right)^{\frac{1}{2}} \\
& \qquad \cdot \sum_{l \leq 1} \left(\iint_{\D^{+}_{t, u} } \frac{\Lambdab(\ub)}{\Lambda (u^\prime)} |L^2 \p^l_x \phi_i|^2 \sqrt{g} \di \tau \di x  \right)^{\frac{1}{2}} \\
 \lesssim {} & (I^2_{N+1} + \delta M^4)  \int^{u}_{-\infty} \frac{1}{\Lambda (u^\prime)} \di u^\prime \int_{C^t_{u^\prime}} \Lambdab(\ub) \sum_{l \leq N} |L \phi_l|^2 \di \ub,
\end{align*}
where we have used  \eqref{energy-bound-Eb-Fb}-\eqref{improved-L-infty-Lb} in the first and third inequalities.
And
\begin{align*}
& \iint_{\D^{+}_{t, u}} \Lambdab (\ub) |L \phi_k|  \left( \frac{I^2_{N+1} + \delta M^4}{\Lambda (u^\prime)}  | L^2 \phi_q| +  \frac{\delta M^2}{\Lambdab^{\frac{1}{2}} (\ub) \Lambda^{\frac{1}{2}} (u^\prime)} (  |L \phi_p|  + |L \Lb \phi_q| ) \right) \sqrt{g} \di \tau \di x  \\
\lesssim {} &   \int^{u}_{-\infty} \left( \frac{I^2_{N+1} + \delta M^4}{\Lambda (u^\prime)} +  \frac{\delta M^2}{ \Lambda^{\frac{1}{2}} (u^\prime)} \right) \di u^\prime \int_{C^t_{u^\prime}} \Lambdab (\ub) \sum_{l \leq k} |L \phi_l|^2 \sqrt{g}  \di \ub  \\
\lesssim {} &   \int^{u}_{-\infty}  \frac{I^2_{N+1} + \delta M^4}{\Lambda^{\frac{1}{2}} (u^\prime)}   \di u^\prime \int_{C^t_{u^\prime}} \Lambdab (\ub) \sum_{l \leq k} |L \phi_l|^2 \sqrt{g} \di \ub  + \delta^3 M^4.
\end{align*}
And the remaining one is a lower order term,
\begin{align}
& \iint_{\D^{+}_{t, u}} \Lambdab (\ub) |L  \phi_k| \left(|\Lb^2 \phi_q| + |\Lb \phi_p| \right) \frac{\delta^2 M^2}{\Lambdab (\ub)} \sqrt{g} \di \tau \di x \nnb \\
 \lesssim {}& \iint_{\D^{+}_{t, u}} \delta^2 M^2 |L  \phi_k| \sum_{l \leq k} |\Lb \phi_l| \sqrt{g} \di \tau \di x  \quad \text{analogous to \eqref{terms-L-Lb} } \nnb \\
 \lesssim {} & \left( \int^{u}_{-\infty} \frac{\delta^2 M^2}{\Lambda (u^\prime)} \di u^\prime \int_{C^t_{u^\prime}} \Lambdab (\ub) |L \phi_k|^2 \di \ub \right)^{\frac{1}{2}}  \left( \int^{+\infty}_{-\infty} \frac{\delta^2 M^2}{\Lambdab (\ub^\prime)} \di \ub^\prime \int_{\Cb^t_{\ub^\prime}} \Lambda (u) \sum_{l \leq k} |\Lb \phi_l|^2 \di u \right)^{\frac{1}{2}} \nnb \\
\lesssim {}& \delta^3 M^4. \label{terms-L-Lb-1}
\end{align}
Now, we turn to the $\Box_{g(\p\phi)}\phi_k \cdot \Lambdab (\ub) (L \phi)^2 \Lb \phi_k$ part.
\begin{align*}
& \iint_{\D^{+}_{t, u}} \Lambdab (\ub) |L \phi|^2 |\Lb \phi_k| \left( \frac{I^2_N + \delta M^4}{\Lambda (u)}  | L^2 \phi_q| +  \frac{\delta M^2}{\Lambdab^{\frac{1}{2}} (\ub) \Lambda^{\frac{1}{2}} (u)} (  |L \phi_p|  + |L \Lb \phi_q| ) \right)  \sqrt{g} \di \tau \di x \\
\lesssim {} & \iint_{\D^{+}_{t, u}} \delta^2 M^4  |\Lb \phi_k|  \sum_{i \leq k}|L \phi_i| \sqrt{g} \di \tau \di x  \quad \text{by \eqref{terms-L-Lb-1} } \\
\lesssim {}& \delta^3 M^6.
\end{align*}
And
\begin{align*}
& \iint_{\D^{+}_{t, u}} \Lambdab (\ub) |L \phi|^2 |\Lb \phi_k| \left(|\Lb \phi_p L^2 \phi_i \Lb \phi_j| + \frac{\delta^2 M^2}{\Lambdab (\ub) } ( |\Lb \phi_p| + |\Lb^2 \phi_q|  ) \right) \sqrt{g} \di \tau \di x \\
 \lesssim {} & \iint_{\D^{+}_{t, u}} \delta^2 M^2 |\Lb \phi_k| \frac{\delta M^2}{\Lambdab^{\frac{1}{2}}(\ub)} \sum_{l \leq k} | \Lb \phi_l| \sqrt{g} \di \tau \di x \lesssim \delta^3 M^6.
\end{align*}

{\bf The deformation tensor:} $\T^{\alpha}_{\beta} [\phi_k] \cdot \p_{\alpha}\TL^{\beta} - \frac{1}{2\sqrt{g}}\TL(\sqrt{g}g^{\gamma\rho}) \p_{\gamma}\phi_k \p_{\rho}\phi_k$.

For $\T^{\alpha}_{\beta} [\phi_k] \cdot \p_{\alpha}\TL^{\beta}$, there is, in views of Lemma \ref{lem-deformation},
\begin{align*}
|\T^{\alpha}_{\beta} [\phi_k] \p_{\alpha} \TL^{\beta}| \lesssim {} & \left(  |\Lb \phi_k|^2 + |L \phi \Lb \phi| |\Lb \phi_k|^2 +  |\Lb \phi|^2 | L \phi_k \Lb \phi_k |  \right) | \Lambdab^\prime (\ub)| |L \phi|^2 \\
&+ \left(  |\Lb \phi_k|^2 + |L \phi \Lb \phi| |\Lb \phi_k|^2 +  |\Lb \phi|^2 | L \phi_k \Lb \phi_k |  \right) \Lambdab (\ub) |L^2 \phi L \phi|  \\
& + \left( |\Lb \phi|^2 |L \phi_k|^2 +  |L \phi|^2 |\Lb \phi_k|^2  \right)  \Lambdab (\ub) | L \Lb  \phi L \phi|.
\end{align*}
It suffices to check the following terms. In analogy with \eqref{esti-leading-Lb} and \eqref{esti-crucial-Lb}, the leading terms can be manipulated as follows.
\begin{align*}
& \iint_{\D^{+}_{t, u}} |\Lb \phi_k|^2 |\Lambdab^\prime (\ub)| |L \phi|^2  \sqrt{g} \di \tau \di x \leq  \iint_{\D^{+}_{t, u}} \Lambdab (\ub) |\Lb \phi_k|^2  |L \phi|^2  \sqrt{g} \di \tau \di x  \\
\lesssim &{} \sup_\tau \int_{\Sigma^{+}_{\tau, u}} \Lambda (u^\prime) |\Lb \phi_k|^2 \sqrt{g} \di x  \int_0^t \Big\|\frac{\Lambdab^{\frac{1}{2}}(\ub)}{\Lambda^{\frac{1}{2}}(u^\prime)} L \phi\Big\|^2_{L^\infty (\Sigma^{+}_{\tau, u})} \di \tau \\
\lesssim &{} \sup_\tau \Eb^2_{k+1}(\tau) \iint_{\D^{+}_{t, u}} \frac{\Lambdab(\ub)}{\Lambda(u^\prime)} (|L \phi|^2 + |L \p_x \phi|^2) \sqrt{g} \di \tau \di x \\
\lesssim &{}  \int^{u}_{-\infty} \frac{I_{N+1}^2 + \delta M^4}{\Lambda(u^\prime)} F^2_{N+1} (u^\prime, t) \di u^\prime,
\end{align*}
and
\begin{align}
& \iint_{\D^{+}_{t, u}} |\Lb \phi_k|^2 \Lambdab (\ub) |L^2 \phi L \phi|  \sqrt{g} \di \tau \di x  \nnb \\
\lesssim &{} \sup_\tau \int_{\Sigma^{+}_{\tau, u}} \Lambda (u^\prime) |\Lb \phi_k|^2  \sqrt{g} \di x  \int_0^t \Big\|\frac{\Lambdab^{\frac{1}{2}}(\ub)}{\Lambda^{\frac{1}{2}}(u^\prime)} L \phi\Big\|_{L^\infty (\Sigma^{+}_{\tau, u})} \Big\|\frac{\Lambdab^{\frac{1}{2}}(\ub)}{\Lambda^{\frac{1}{2}}(u^\prime)} L^2 \phi\Big\|_{L^\infty (\Sigma^{+}_{\tau, u})} \di \tau \nnb \\
\lesssim &{} \sup_\tau \Eb^2_{k+1}(\tau) \int_0^t \left( \Big\|\frac{\Lambdab^{\frac{1}{2}}(\ub)}{\Lambda^{\frac{1}{2}}(u^\prime)} L \phi\Big\|^2_{L^\infty (\Sigma^{+}_{\tau, u})} +\Big\|\frac{\Lambdab^{\frac{1}{2}}(\ub)}{\Lambda^{\frac{1}{2}}(u^\prime)} L^2 \phi\Big\|^2_{L^\infty (\Sigma^{+}_{\tau, u})} \right) \di \tau \nnb \\
\lesssim &{}  (I_{N+1}^2 + \delta M^4)  \sum_{i \leq 1} \iint_{\D^{+}_{t, u}} \frac{\Lambdab(\ub)}{\Lambda(u^\prime)} (|L \p^i_x \phi|^2 + |L^2 \p^i_x \phi|^2)  \sqrt{g} \di \tau \di x \nnb \\
\lesssim &{} \int^{u}_{-\infty} \frac{I_{N+1}^2 + \delta M^4}{\Lambda(u^\prime)} F^2_{N+1} (u^\prime, t) \di u^\prime. \label{esti-crucial-L}
\end{align}
We remark that the improvement \eqref{energy-bound-Eb-Fb} is crucial in the above estimates.
The rest are all of lower order,
\begin{align*}
& \iint_{\D^{+}_{t, u}} |\Lb \phi|^2 |L \phi_k \Lb \phi_k| \left( \Lambdab (\ub) |L^2 \phi L \phi| + |\Lambdab^\prime (\ub)| |L \phi|^2 \right) \sqrt{g} \di \tau \di x \\
 \lesssim{}&  \iint_{\D^{+}_{t, u}} \frac{M^2}{\Lambda(u)} |L \phi_k \Lb \phi_k| \delta^2 M^2 \sqrt{g} \di \tau \di x  \lesssim \delta^3 M^6,
\end{align*}
where the estimate in the final inequality is carried out in the same way as \eqref{terms-L-Lb-1},
and
\begin{align*}
& \iint_{\D^{+}_{t, u}} \left( |\Lb \phi|^2 |L \phi_k|^2 +  |L \phi|^2 |\Lb \phi_k|^2  \right)  \Lambdab (\ub) |\Lb L \phi L \phi|  \sqrt{g} \di \tau \di x  \\ \lesssim {} &  \iint_{\D^{+}_{t, u}} \left( \frac{M^2}{\Lambda(u^\prime)} |L \phi_k|^2 \delta^2 M^2 +  \delta^2 M^2 |\Lb \phi_k|^2  \frac{\delta M^2}{\Lambdab^{\frac{1}{2}}(\ub) \Lambda^{\frac{1}{2}}(u^\prime)}\right)  \sqrt{g} \di \tau \di x  \lesssim  \delta^3 M^6.
\end{align*}

For $\frac{1}{2\sqrt{g}}\TL(\sqrt{g} g^{\gamma\rho}) \p_{\gamma}\phi_k \p_{\rho}\phi_k$, after a similar calculation as before, we obtain,
\begin{align*}
& \frac{1}{2\sqrt{g}}\TL(\sqrt{g}g^{\gamma\rho}) \p_{\gamma}\phi_k \p_{\rho}\phi_k \\ 
={}& - \frac{1}{2g} ( \TL \Lb \phi \Lb \phi |L \phi_k|^2 + \TL L \phi L \phi |\Lb \phi_k|^2) + \frac{ \TL \p_\mu \phi \p^\mu \phi}{g^2} \p^\gamma \phi \p^\rho \phi \p_{\gamma}\phi_k \p_{\rho}\phi_k.
\end{align*}

Since $\TL = \Lambdab (\ub) L + \Lambdab (\ub) |L \phi|^2 \Lb$, and $\Lambdab (\ub) L \p \phi$ is always larger than $\Lambdab (\ub) |L \phi|^2 \Lb \p \phi$, it suffices to check the following terms.
\begin{align*}
& \iint_{\D^{+}_{t, u}} \Lambdab (\ub) |L\Lb \phi \Lb\phi| |L \phi_k|^2 \sqrt{g} \di \tau \di x  \\
= {}& \iint_{\D^{+}_{t, u}} \Lambdab (\ub) |\Lb L \phi \Lb\phi| |L \phi_k|^2 \sqrt{g} \di \tau \di x\\
\lesssim {}&  \iint_{\D^{+}_{t, u}} \frac{I_{N+1}^2 + \delta M^4}{\Lambda(u)} \Lambdab (\ub) |L \phi_k|^2 \sqrt{g} \di \tau \di x  \\
\lesssim {} & \int^{u}_{-\infty} \frac{I_{N+1}^2 + \delta M^4}{\Lambda(u)} F^2_{k+1} (u^\prime, t) \di u^\prime.
\end{align*}
And $$\iint_{\D^{+}_{t, u}} |\Lb \phi_k|^2 \Lambdab (\ub) |L^2 \phi L \phi| \sqrt{g} \di \tau \di x  \lesssim \int^{u}_{-\infty} \frac{I_{N+1}^2 + \delta M^4}{\Lambda(u^\prime)}  F^2_{N+1} (u^\prime, t) \di u^\prime$$ is already shown in \eqref{esti-crucial-L}.
The last one $\frac{ \TL \p_\mu \phi \p^\mu \phi }{g^2} \p^\gamma \phi \p^\rho \phi \p_{\gamma}\phi_k \p_{\rho}\phi_k$ is a lower order term obeying the estimate $$\iint_{\D^{+}_{t, u}} \frac{ |\TLb \p_\mu \phi \p^\mu \phi|}{g^2} |\p^\gamma \phi \p^\rho \phi \p_{\gamma}\phi_k \p_{\rho}\phi_k|  \sqrt{g} \di \tau \di x \lesssim \delta^3 M^6.$$

In the end, we arrive at the conclusion
\bel{eq-energy-estimate-TL}
E^2_{N+1}(u, t) + F^2_{N+1}(u, t)  \lesssim \delta^2 I^2_{N+1} + \delta^3 M^6 + \int^{u}_{-\infty} \frac{I_{N+1}^2 + \delta M^4}{\Lambda^{\frac{1}{2}}(u^\prime)} F^2_{N+1} (u^\prime, t) \di u^\prime,
\ee
which further implies
\bes
F^2_{N+1}(u, t)  \lesssim \delta^2 I^2_{N+1} + \delta^3 M^6 + \int^{u}_{-\infty} \frac{I_{N+1}^2 + \delta M^4}{\Lambda^{\frac{1}{2}}(u^\prime)} F^2_{N+1} (u^\prime, t) \di u^\prime.
\ees
The Gronwall's inequality leads to
\bes
F^2_{N+1} (u, t)  \lesssim \exp(I_{N+1}^2 + \delta M^4) (\delta^2 I^2_{N+1} + \delta^3 M^6),
\ees
Substituting it into \eqref{eq-energy-estimate-TL}, we achieve
\bel{Energy-Bound-E-F}
E^2_{N+1}(u, t) + F^2_{N+1}(u, t)  \lesssim  \delta^2 I^2_{N+1} + \delta^3 M^6.
\ee

\section{Initial data}\label{Sec:ID}
The aim for us in this section is to show the following main result with regard to the initial data.

Let us introduction a definition for convenience. For any function $f \in C^\infty(\mathbb R)$, we call $f(x) \in O_\gamma$, if
\bes
\int_{\mathbb R}(1+|x|)^{2+2\gamma}|f^{(k)} (x)|^2dx \lesssim 1, \quad \text{for all integer} \,\, k \geq 0,
\ees
recalling that $f^{(k)} (x) = \frac{\di^{k}}{\di x^{k}} f(x)$.
\begin{theorem}\label{thm3}
Suppose our data set satisfies the assumptions in Theorem \ref{main-theorem} and $\phi$ is the solution of Cauchy problem, then it follows that
\be\label{ID1}
L\phi_{k}|_{t=0} = \delta f_k(x),\quad \text{and}\quad \underline{L}\phi_{k}|_{t=0} = \fb_k(x),
\ee
where $f_{k}(x), \, \fb_{k}(x) \in O_\gamma$ are smooth functions depending only on $(F(x),G(x))$.
\end{theorem}
\begin{proof}
We will prove the theorem by induction. When $k=0$, we have
\begin{align}
L\phi|_{t=0} & =(\phi_t+\phi_x)|_{t=0}=G(x)+F^{\prime}(x) = \delta f(x), \label{ID2} \\
\Lb \phi|_{t=0}&=(\phi_{t}-\phi_{x})|_{t=0}=G(x)-F^{\prime}(x) = \fb(x). \label{ID3}
\end{align}
Letting $f_0 (x)=f(x), \, \fb_0 (x) = \fb (x)$, we know that $f(x), \, \fb(x) \in O_\gamma$ and hence \eqref{ID1} holds true.

To proceed to higher order cases, it will be useful to work within the null frame $(L,\,\underline{L})$ (instead of the Cartesian one $(\p_t, \, \p_x)$). We rewrite the relativistic string equation as
\be\label{ID4}
2(2-L\phi\underline{L}\phi)L\underline{L}\phi+(\underline{L}\phi)^2L^2\phi+(L\phi)^2\underline{L}^2\phi=0.
\ee
And its higher order versions take the form of
\begin{align}\label{ID5}
&2(2-L\phi\underline{L}\phi)L\underline{L}\phi_k+(\underline{L}\phi)^2L^2\phi_k+(L\phi)^2\underline{L}^2\phi_k\nonumber\\
=-&\sum_{l=0}^{k-j}\sum_{j=0}^{k-1}C_{k-j}^{l}C_{k}^{j} \left(-2L\phi_l\underline{L}\phi_{k-j-l}L\underline{L}\phi_j
+2\underline{L}\phi_l\underline{L}\phi_{k-j-l}L^2\phi_j+2L\phi_lL\phi_{k-j-l}\underline{L}^2\phi_j\right),
\end{align}
where $C_{k-j}^{l}, \, C_{k}^{j}$ are some constants.

Suppose that \eqref{ID1} holds true for all $n\leq k$, that is,
\bes
L\phi_{n}|_{t=0}=\delta f_{n}(x) \in \delta O_\gamma, \quad \text{and}\quad \underline{L}\phi_{n}|_{t=0}=\fb_{n}(x) \in O_\gamma.
\ees
Then for $\phi_{k+1}$, we have
\begin{align}
L\partial_x\phi_{k}|_{t=0} & =\partial_{x}L\phi_{k}|_{t=0}=\delta  f^\prime_k(x) \in \delta O_\gamma, \label{ID6} \\
\underline{L}\partial_{x}\phi_{k}|_{t=0} & =\partial_{x}\underline{L}\phi_{k}|_{t=0}= \fb^\prime_{k}(x) \in O_\gamma. \label{ID7}
\end{align}
which shows that \eqref{ID1} holds for $L\partial_x\phi_{k}|_{t=0}$ and $\underline{L}\partial_{x}\phi_{k}|_{t=0}$.

Observe that the right hand side of \eqref{ID5} contains only lower order derivative terms, like $L\phi_{n}$ and $\underline{L}\phi_{n}$ for $n\leq k$, and each term contains at least one $L\phi_n$, which will afford a $\delta$ factor. We can substitute \eqref{ID1} for the $n\leq k$ case into the right hand side of \eqref{ID5}, so that
\be\label{ID8}
2(2-L\phi\underline{L}\phi)L\underline{L}\phi_k|_{t=0}+(\underline{L}\phi)^2L^2\phi_k|_{t=0}
+(L\phi)^2\underline{L}^2\phi_k|_{t=0}=\delta N_{k}(x) \in \delta O_\gamma,
\ee
where $N_{k}(x)$ is a smooth function depending only on the initial data. Noticing that
\begin{align}
L^2\phi_k|_{t=0} & =L(\partial_{t}-\partial_x+2\partial_x)\phi_k|_{t=0}=L\underline{L}\phi_k|_{t=0}+2L\partial_{x}\phi_k|_{t=0}, \label{ID9} \\
\underline{L}^2\phi_k|_{t=0} & =\underline{L}(\partial_{t}+\partial_x-2\partial_x)\phi_k|_{t=0}=L\underline{L}\phi_k|_{t=0}-2\underline{L}\partial_x\phi_k|_{t=0}, \label{ID10}
\end{align}
we substitute \eqref{ID6}-\eqref{ID7}, \eqref{ID9}-\eqref{ID10} into \eqref{ID8}, and obtain
\begin{align}\label{ID11}
L\underline{L}\phi_{k}|_{t=0}&=\frac{\delta N_{k}(x)-2(\underline{L}\phi)^2L\partial_{x}\phi_k+2(L\phi)^2\underline{L}\partial_{x}\phi_k}{4+(\underline{L}\phi)^2+(L\phi)^2-2L\phi\underline{L}\phi}\Big|_{t=0}\nonumber\\
&=\frac{\delta \left( N_{k}(x)-2 \fb^2_{0}(x) f^\prime_{k}(x)+2\delta f^2_{0}(x)  \fb^\prime_{k}(x) \right)}{4+\fb^2_{0}(x)-2\delta \fb_{0}(x)f_{0}(x)+\delta^2f^2_{0}(x)} \in \delta O_\gamma,
\end{align}
if $\delta$ is sufficiently small.
Consequently, we have
\begin{align}
L\partial_{t}\phi_k|_{t=0} & =L\underline{L}\phi_{k}|_{t=0}+L\partial_{x}\phi_{k}|_{t=0} \in \delta O_{\gamma}, \label{ID12} \\
\underline{L}\partial_{t}\phi_{k}|_{t=0} & =\underline{L}L\phi_{k}|_{t=0}-\underline{L}\partial_{x}\phi_k|_{t=0}\in O_{\gamma}. \label{ID13}
\end{align}

Combining \eqref{ID6}-\eqref{ID7} with \eqref{ID12}-\eqref{ID13}, we can see that \eqref{ID1} holds for $\phi_{k+1}$. We complete the proof of Theorem \ref{thm3}.
\end{proof}

\noindent{\Large {\bf Acknowledgements.}} J.W. is supported by the NSF of Fujian Province (Grant No. 2018J05010) and NSFC (Grant No. 11701482). C.W. is supported by the NSFC (Grant No. 12071435, 11871212), the NSF of Zhejiang Province (Grant No. LY20A010026) and the Fundamental Research Funds of Zhejiang Sci-Tech University (Grant No. 2020Q037).

\end{document}